\documentclass{amsart}

\usepackage{amssymb}
\usepackage{amsthm}
\usepackage{mathrsfs}
\usepackage{enumerate} 
\usepackage{color}

\usepackage{chngcntr}   

\usepackage{rotating}

\usepackage[all]{xy}

\usepackage{fancyhdr}
\usepackage{xr-hyper}

\usepackage[colorlinks=true]{hyperref}

\newtheorem{theorem}{Theorem}[section] 
\newtheorem{proposition}[theorem]{Proposition}
\newtheorem{lemma}[theorem]{Lemma}

\newtheorem{corollary}[theorem]{Corollary}

\theoremstyle{definition}
\newtheorem{definition}[theorem]{Definition}

\newtheorem{notation}[theorem]{Notation}

\theoremstyle{remark}
\newtheorem{remark}[theorem]{Remark}

\numberwithin{equation}{section}

\newcounter{myenum}

\makeatletter
\renewcommand\thetheorem{\@arabic\c@section.\@arabic\c@theorem}
\newcounter{subtheorem}
 \@addtoreset{subtheorem}{theorem}
\renewcommand\thesubtheorem{\thetheorem.\@arabic\c@subtheorem}

\makeatother

\newcommand{\cal}{\mathcal}
\newcommand{\bb}{\mathbb}

\newcommand{\comment}[1]{}
\newcommand{\noi}{\noindent}
\newcommand{\xym}{\xymatrix}

\newcommand{\into}{\hookrightarrow}
\newcommand{\PP}{\bb{P}}
\newcommand{\AAA}{\bb{A}}
\newcommand{\A}{\mathbb{A}^1}

\DeclareMathOperator{\pic}{Pic}
\DeclareMathOperator{\Hom}{Hom}

\DeclareMathOperator{\spec}{Spec}
\DeclareMathOperator{\id}{id}

\DeclareMathOperator{\Sm}{{\bf{Sm}}}
\DeclareMathOperator{\Sch}{{\bf{Sch}}}

\newcommand{\oline}{\overline}
\DeclareMathOperator{\coker}{coker}

\begin{document}

\title{Cohomology of presheaves with oriented weak transfers}
\author{Joseph Ross}
\address{University of Southern California Mathematics Department, 3620 South Vermont Avenue, Los Angeles, California}
\urladdr{\url{http://www-bcf.usc.edu/~josephr/}}
\email{josephr@usc.edu}

%


\date{\today}

\maketitle
\begin{abstract} Over a field of characteristic zero, we establish the homotopy invariance of the Nisnevich cohomology of homotopy invariant presheaves with oriented weak transfers, and the agreement of Zariski and Nisnevich cohomology for such presheaves.  This generalizes a foundational result in Voevodsky's theory of motives.  The main idea is to find explicit smooth representatives of the correspondences 
which provide the input for Voevodsky's cohomological architecture. 
\end{abstract}

\tableofcontents

\section{Introduction}
The goal of the theory of motives is the construction of a universal cohomology theory for algebraic varieties.  While an abelian category of mixed motives has not yet been constructed, Voevodsky has defined a triangulated category of effective motivic complexes over a field $k$, denoted $DM^{eff}_-(k)$, which has the properties one would expect of the derived category of motives \cite{TRICAT}.  Moreover, this category has proved useful in its own right, for example in Voevodsky's celebrated proof of the Milnor conjecture.

This approach succeeds in large part by transforming questions about the geometric category
$DM^{eff}_-(k)$ into questions about a category of a more homological nature.  
First, $DM^{eff}_-(k)$ is a subcategory of $D^-(Sh_{Nis}(SmCor(k)))$, a category of complexes of sheaves.  
Now the crucial ingredient is Voevodsky's functor $\textbf{R}C : D^-(Sh_{Nis}(SmCor(k))) \to DM^{eff}_-(k)$ which is left adjoint to the inclusion and identifies  $DM^{eff}_-(k)$ with the localization of $D^-(Sh_{Nis}(SmCor(k)))$ at a collection of morphisms expressing the invertibility of the affine line.
The construction of $\textbf{R}C$ itself depends critically on the cohomological theory of presheaves with transfers, in particular the result that (over a perfect field $k$), the Zariski cohomology and Nisnevich cohomology of a homotopy invariant presheaf with transfers coincide, and this cohomology is itself homotopy invariant \cite{CTPST}.  The main result of this paper generalizes Voevodsky's cohomological result to a class of presheaves with a weaker form of transfer structure.

\begin{theorem}[\ref{PSOWT}, \ref{cohomology is homotopy invariant}, \ref{Zar = Nis}]
Let $k$ be a field of characteristic zero, and let $\cal{F}$ be a homotopy invariant presheaf of abelian groups on $\Sm / k$ with oriented weak transfers for affine varieties.  Then the Nisnevich cohomology presheaves $H^n_{Nis}(-, \cal{F}_{Nis})$ are homotopy invariant, and Zariski cohomology coincides with Nisnevich cohomology for $\cal{F}$.
\end{theorem}

Our transfers are very close to the weak transfers considered by Panin-Yagunov \cite{PanYagRigid}, Yagunov \cite{YagRigid}, and Hornbostel-Yagunov \cite{HornYag}.  Motivation for our definition comes from oriented $T$-spectra, but our definition and results are independent of the theory of $T$-spectra.  See Definition \ref{PSOWT} for details.

The reader will see that we rely quite heavily on Voevodsky's architecture.  Our strategy is simply to make explicit the geometric input ($\A$-families of correspondences) of the cohomological theory, and to investigate ``how nice" these correspondences can be made.  In short, we find that all of the necessary constructions can be realized as finite morphisms between smooth affine schemes with trivial sheaf of K\"ahler differentials.  See Subsection \ref{main idea} for a more detailed description of the main idea.  Along the way, we prove results about presheaves with (oriented) weak transfers which may be of independent interest: vanishing on semilocal schemes is detected at the generic point (Corollary \ref{semilocal inj}), the Mayer-Vietoris sequence for open subschemes of the affine line (Theorem \ref{open affine line MV}), Nisnevich excision (Theorem \ref{general Nis invariance}), and the Gersten resolution (Theorem \ref{Gersten}).  At some places we need the field $k$ to be infinite, but the more serious reason for the restriction on the characteristic of $k$ is our use of a Bertini theorem.

The basic example of a presheaf with transfers is a presheaf of cycles \cite{SV}; the basic non-example is algebraic $K$-theory.  (Examples are discussed in further detail in Subsection \ref{examples}.)
Adaptations of Voevodsky's argument to a wider class of theories, especially with the goal of comparing algebraic cycles and $K$-theory, have appeared before.
Mark Walker extended much of Voevodsky's theory to $K_0$-presheaves, a context which includes algebraic $K$-theory, and proved that two definitions of motivic cohomology (one cycle-theoretic, the other $K$-theoretic) are equivalent.
In their construction of the motivic spectral sequence, Friedlander and Suslin used the notion of pseudo-pretheory (a generalization of Voevodsky's notion of pretheory) to obtain fibrations of $K$-theory spectra with supports  over semilocal schemes from such fibrations initially defined over fields \cite{FS}.
Combining Voevodsky's arguments with known results about Witt groups, Panin extended the cohomological results (homotopy invariance and the coincidence of Zariski and Nisnevich cohomology) to the Nisnevich sheafification of the Witt groups 
\cite{PanWNis}.  It is interesting to note that Witt theory is not orientable.
More recently,  Heller, {\O}stv{\ae}r, and Voineagu developed an equivariant version of presheaves with transfers \cite{HOV}.   Other recent generalizations of Voevodsky's construction of motives, namely reciprocity sheaves and presheaves with traces, are discussed briefly in Subsection \ref{context}.

We felt there was some intrinsic interest in isolating and making explicit the geometric aspects of \cite{CTPST}.  Additionally, we imagine that the geometric constructions presented here might prove useful to future adaptations of \cite{CTPST} (and hence \cite{TRICAT}).

\textbf{Acknowledgments.} The results of this paper were obtained in 2010-11 while the author was a wissenschaftlicher Mitarbeiter at the Universit\"at Duisburg-Essen.  The author wishes to thank Marc Levine for the suggestion to investigate whether the geometric constructions of Voevodsky could be adapted to the smooth setting, and for many helpful discussions.  The author also wishes to thank Aravind Asok, Jeremiah Heller, and Anastasia Stavrova for their interest and encouragement.

\section{Presheaves with (oriented weak) transfers}

\subsection{Definition and basic properties}
For a field $k$, we denote by $\textbf{Sm} / k$ the category of smooth separated $k$-schemes of finite type.

\begin{definition} \label{PSOWT} Let $A$ be an additive category, and $\cal{F} :  {(\textbf{Sm} / k)}^{op} \to A$ a functor.  We assume $\cal{F}$ is additive in the sense that $\cal{F}(X_1 \coprod X_2) \cong \cal{F}(X_1) \oplus \cal{F}(X_2)$.  We say the presheaf $\cal{F}$ is \textit{homotopy invariant} if the canonical map ${pr_1}^* : \cal{F}(X) \to \cal{F}(X \times \A)$ is an isomorphism for all $X \in \Sm /k$.  We say $\cal{F}$ has \textit{weak transfers} if for any $X, Y \in \Sm /k$, and any closed embedding (of $Y$-schemes) $X \into Y \times \bb{A}^n$ such that $f :X \to Y$ is finite, flat, and generically \'{e}tale, and so that the normal bundle $N_X (Y \times \mathbb{A}^n)$ is trivial (and trivialized via $\psi$), we are given maps $f_\ast^\psi: \cal{F}(X)  \to \cal{F}(Y)$ satisfying the following properties.

\begin{enumerate}
\item The $f_\ast$'s are compatible with disjoint unions: if $X =X_1 \coprod X_2$ and $f_i : X_i \to Y$ denotes the induced morphism, then the diagram:
$$\xym{ \cal{F}(X) \ar[d]_-\cong \ar[r]^-{f_\ast} & \cal{F}(Y) \\ \cal{F}(X_1) \oplus \cal{F}(X_2) \ar[ru]_{ {f_1}_{\ast} , {f_2}_{\ast} } }$$
commutes.  We have suppressed the $\psi$'s since there is a canonical isomorphism $N_X(Y \times \bb{A}^n) = N_{X_1}(Y \times \bb{A}^n) \oplus N_{X_2}(Y \times \bb{A}^n)$.

\item The $f_\ast$'s are compatible with sections $s : Y \to X$ which are isomorphisms onto connected components of $X$.  In the notation of the previous property, supposing $s$ is an isomorphism onto $X_1$, then we require ${f_1}_\ast = s^\ast$ (for any embedding and trivialization).

\item The $f_\ast^\psi$'s are functorial for embeddings $g: Y' \into Y$  of principal smooth divisors such that $X' := X \times_Y Y' \in \Sm /k$.  That is, given such a $g$, the diagram:
$$\xym{ \cal{F}(X)  \ar[r]^-{{g'}^*} \ar[d]_-{{f}^\psi_\ast} & \cal{F}(X') \ar[d]^-{{f'}^{\psi'}_{\ast}} \\ \cal{F}(Y) \ar[r]^-{g^*} & \cal{F}(Y') }$$
commutes.  Corollary \ref{prin divisor} implies this diagram makes sense and defines $\psi'$.

\item The $f_\ast^\psi$'s are functorial for smooth morphisms $g : Y' \to Y$.

\item The $f_*^\psi$'s are compatible with the addition of irrelevant summands: suppose $i =(f,\beta) : X \into Y \times \bb{A}^n$ is an $N$-trivial embedding, with the normal bundle trivialized via $\psi$.  Let $(i, 0) =(f,0,\beta) : X \into Y \times \A \times \bb{A}^n$.  Then, using the notation of Lemma \ref{N trivial A1 stable}, there is a canonical identification $N^\vee_i \oplus \cal{O} = N^{\vee}_{(i,0)}$.  Hence $\psi$ induces a trivialization $(\psi,0)$ of $N_{(i,0)}$.  Then the requirement is that $f_*^\psi = f_*^{(\psi,0)} : \cal{F}(X) \to \cal{F}(Y)$.

\end{enumerate}

If $A$ is the category $\bf{Ab}$ of abelian groups, then we require all of the weak transfers $f_*$ to be homomorphisms of abelian groups.  We say $\cal{F}$ has \textit{oriented} weak transfers if every map $f_*^\psi$ is independent of the trivialization $\psi$.  We say $\cal{F}$ has \textit{weak transfers for affine varieties} if we are given weak transfers as described above whenever $X$ and $Y$ are affine. 
A \textit{morphism} of presheaves with some kind of transfer structure is a morphism of presheaves compatible with all transfer maps.
\end{definition}

\begin{remark} Our notion of weak transfers is very close to the notion of weak transfers studied by Panin-Yagunov \cite{PanYagRigid} and Yagunov \cite{YagRigid}.  Our condition (1) is exactly the additivity condition, and our conditions (3) and (4) identify the particular types of transversal base change needed in the proof.  Our condition (2) is a stronger form of the normalization condition of \cite[Property 1.7]{PanYagRigid}, \cite[Proposition 3.3]{YagRigid}.  Actually, we need condition (2) only for certain $X \in \Sm / k$ (e.g., open subschemes of the affine line over a local scheme) which appear in various constructions.  
Our condition (5) is not used by these authors; it is motivated by the weak transfers in $T$-spectra, and it is used in the proof of ``independence of embedding" below.
Rigidity theorems for presheaves with weak transfers are established in \cite{PanYagRigid}, \cite{YagRigid}, and \cite{HornYag}. \end{remark}

\textbf{Basic properties.} Suppose $A$ is an abelian category.  Then the category of homotopy invariant presheaves with ((oriented) weak) transfers (for affine varieties) with values in $A$ is abelian; and the inclusion of this category into the category of $A$-valued presheaves with ((oriented) weak)  transfers (for affine varieties) is exact.  Any subpresheaf with ((oriented) weak)  transfers (for affine varieties) of a homotopy invariant presheaf with the same transfers is homotopy invariant.  Any presheaf with transfer extension of homotopy invariant presheaves with the same transfer structure is homotopy invariant.

\subsection{Consequences of orientation} Under the hypothesis of homotopy invariance, if the weak transfers are independent of the normal bundle trivialization, then they are independent of the $N$-trivial embedding.  This justifies the omission of the map $X \to \bb{A}^n$ in the notation for the weak transfer.  For clarity in the proof, we will decorate the weak transfer with further notation for the map $X \to \bb{A}^n$.

\begin{lemma}[independence of embedding] \label{ind of embedding} Suppose $\cal{F}$ is a homotopy invariant presheaf on $\Sm / k$ with oriented weak transfers, and let $f : X \to Y$ be a finite, generically \'etale morphism in $\Sm/k$.  Suppose $\Omega^1_{X/k}$ and $\Omega^1_{Y/k}$ are trivial.  Then the map $f_* : \cal{F}(X) \to \cal{F}(Y)$ is independent of the choice of $N$-trivial embedding. \end{lemma}

\begin{proof} Let $(f, \alpha) : X \into Y \times \bb{A}^n$ and $(f, \beta) : X \into Y \times \bb{A}^m$ be $N$-trivial embeddings.  Consider the closed immersion $\iota: \A \times X \to \A \times Y \times \bb{A}^n \times \bb{A}^m$ defined by $\iota (t,x) = (t,f(x),t\alpha, (1-t)\beta)$.  Since both $\A \times X$ and $\A \times Y$ have trivial sheaf of K\"ahler differentials, Lemma \ref{N trivial no kahler} implies that the product of $\iota$ with a constant morphism $\A \times X \to \bb{A}^{\dim Y + 1}$ is an $N$-trivial embedding.  We denote the product morphism by $(\iota, c)$.

Homotopy invariance implies the maps $i_0^*, i_1^* : \cal{F}(\A \times X) \to \cal{F}(X)$ are both inverse to the isomorphism induced by the projection, hence are equal.  Now the compatibility of the weak transfers with the inclusions $0 \times Y, 1 \times Y \into \A \times Y$ implies $(1 \times f, \iota)_* = (f, \iota |_{0 \times X})_* = (f, \iota |_{1 \times X})_*$.

We have $\iota(0,x) = (0, f(x), 0, \beta, c)$.  Since the weak transfers are compatible with the addition of irrelevant summands, we have $(f, \iota |_{0 \times X})_* = (f, \beta)_*$.  Similarly, since $\iota(1,x) = (1, f(x), \alpha, 0, c)$ we have $(f, \iota |_{1 \times X})_* = (f, \alpha)_*$.
\end{proof}

Next we have the analogue of \cite[Prop.~3.11]{CTPST}.

\begin{lemma}[factorization through rational equivalence] Let $X,S$  be smooth $k$-schemes, and suppose $f :X \to \A \times S$ is a finite, generically \'etale morphism which admits an $N$-trivial embedding.  Furthermore suppose the base change $X_i := X \times_{\A \times S} i \times S$ is smooth for $i=0,1 \in \A (k)$.  Let $g_0 : X_0 \into X$ and $f_0 : X_0 \to S$ denote the induced morphisms.

\noi Let $\cal{F}$ be a homotopy invariant presheaf on $\Sm / k$ with oriented weak transfers.

\noi Then we have ${f_0}_{\ast}  \circ  {g_0}^{\ast} = {f_1}_{\ast}  \circ  {g_1}^{\ast} : \cal{F}(X) \to \cal{F}(S)$. \end{lemma}

\begin{proof} The compatibility with the transverse squares determined by the embeddings $0 \times S, 1 \times S \into \A \times S$ implies that $i_0^{\ast}  \circ  f_{\ast} =   {f_0}_{\ast}  \circ  {g_0}^{\ast} :  \cal{F}(X) \to \cal{F}(S)$ (and similarly for the fiber at $1 \times S$).  Since $i_0^* = i_1^*$, the result follows. \end{proof}

The analogue of \cite[Prop.~3.12]{CTPST} is immediate because our transfers are defined on the groups $\cal{F}(X)$ rather than as homomorphisms from a group of cycles.  Thus the requirement that $\cal{F}$ be a presheaf gives the following result.

\begin{lemma} Let $f : X \to Y$ be a finite, generically \'etale morphism in $\Sm / k$ which admits an $N$-trivial embedding.  Suppose $i: X \to C$ is a closed immersion into a smooth $Y$-curve $C$, and suppose $j : C \subset C'$ is open immersion of smooth $Y$-curves.

\noi Let $\cal{F}$ be a homotopy invariant presheaf on $\Sm / k$ with oriented weak transfers.

\noi Then we have $f_* \circ {(j \circ i)}^* = f_* \circ i^* \circ j^* : \cal{F}(C') \to \cal{F}(Y)$. \end{lemma}

\subsection{Examples} \label{examples}
Presheaves with transfers in the sense of Voevodsky \cite{CTPST} are presheaves with oriented weak transfers.  More generally, the $K_0$-presheaves considered by Mark Walker \cite{WalkerThesis} have oriented weak transfers.  (One can ignore the embeddings and normal bundles.  Given a finite flat morphism $f : Y \to X$, the transpose of the graph of $f$ determines a class in $K_0(X, Y)$, hence a morphism $f_* :\cal{F}(Y) \to \cal{F}(X)$.   See \cite[Lemma 5.6]{WalkerThesis}.)
To get a feeling for the differences among these, suppose
$\cal{Z} = \cal{Z}_1 \cup \cal{Z}_2$ is a $B$-relative zero-cycle in a smooth curve $C \to B$ with $B \in \Sm / k$.
If $\cal{F}$ is a presheaf with transfers (or a pretheory), then
\begin{equation} \label{transfer eqn} \phi_{\cal{Z}} = \phi_{\cal{Z}_1} + \phi_{\cal{Z}_2} : \cal{F}(C) \to \cal{F}(B). \end{equation}
If $\cal{F}$ is a $K_0$-presheaf (or a pseudo-pretheory), then (\ref{transfer eqn}) holds if the ideal sheaf of $\cal{Z}_1$ is trivial upon restriction to $\cal{Z}_2$ (or vice versa), but in general (\ref{transfer eqn}) may fail to hold.  If $\cal{F}$ is a presheaf with oriented weak transfers, then (\ref{transfer eqn}) holds provided $\cal{Z}_1$ and $\cal{Z}_2$ are themselves smooth (in particular, multiplicity-free), the morphisms $\cal{Z}_i \to B$ admit $N$-trivial embeddings, and $\cal{Z}_1 \cap \cal{Z}_2 = \emptyset$.  Without these conditions, one or both sides of (\ref{transfer eqn}) may not be defined.  In short, in our setting there are fewer transfer morphisms, and it is more difficult to verify relations among them.

Further examples come from $T$-spectra: the $0$-space of a $T$-spectrum $E$ (or its homotopy (pre)sheaves) has weak transfers. 
See \cite{PanYagRigid} or \cite[\S 9]{MLChow} for details on the construction.
The condition (5) in Definition \ref{PSOWT} roughly corresponds to passing to the $0$-space.
That the weak transfers in $T$-spectra satisfy condition (5) boils down to a lemma of Spitzweck \cite[Lemma 3.5]{MSrelns} which asserts there is a canonical isomorphism $T \wedge X / T \wedge (X - Z) = Th (\cal{O}_X) / Th(\cal{O}_{X - Z}) \cong Th (N_Z X \oplus \cal{O}_Z)$ in $\cal{H}_\bullet(k)$.  Here  $\cal{H}_\bullet(k)$ is the homotopy category of the Morel-Voevodsky category of pointed simplicial presheaves on $\Sm / k$ with the $\A$-Nisnevich model structure \cite[Thm.~2.3.2, Defn.~3.2.1]{MV}.

If the $T$-spectrum $E$ is oriented, then a vector bundle automorphism of a vector bundle $V$ over $X \in \Sm / k$ induces the identity map on $E (Th(V))$ \cite[Defn.~3.1.1]{Panin}.  In particular the choice of trivialization of the normal bundle does not influence the weak transfer map, so such a spectrum has oriented weak transfers.  Following Yagunov's observation \cite[p.~30]{YagRigid}, we point out that we only use independence of trivialization for normal bundles arising in our constructions, which is a bit different from the full strength of the orientation.

Additionally, the weak transfers are inherited by various ``support" constructions on $T$-spectra $E$ (possibly after passing to homotopy presheaves), in particular the spectra $E^Q$ and $E^{(q)}$ considered in \cite{CTS}, which are geometric models for the slice tower of $E$.  Note that, in characteristic zero, Levine has shown the higher slices of a (not necessarily orientable) $S^1$-spectrum $E$ have filtrations whose subquotients are complexes of homotopy invariant presheaves with transfers.  For the zero slice, one has a similar result after taking homotopy sheaves of the loop space; see \cite{ST}.
In fact our original motivation for proving homotopy invariance of cohomology was to develop localization machinery applicable to the presheaves of spectra $E^Q, E^{(q)}$ (maybe assuming $E$ oriented) and thereby extend to all quasi-projective varieties the main result of \cite{CTS}.  We encountered difficulties in applying the Friedlander-Lawson moving lemma (as in \cite{BCC}) to a general cohomology theory $E$.

\subsection{Context} \label{context} Nisnevich sheaves of abelian groups with homotopy invariant cohomology are called \textit{strictly $\A$-invariant} by Morel \cite[Defn.~7]{MorelBook}.
A large supply of homotopy invariant cohomology is provided by the following result of Morel:
the homotopy sheaves (in degree $\geq 2$) of a pointed motivic space (i.e., object of $\cal{H}_\bullet(k)$) are strictly $\bb{A}^1$-invariant, and in degree $1$ are strongly $\bb{A}^1$-invariant \cite[Cor.~5.2]{MorelBook}.
Homotopy modules are the motivic analogues of stable homotopy groups \cite[p.~520]{Deglise}.
A theorem of D{\'e}glise identifies the homotopy modules with transfers as those, among homotopy modules, which are orientable \cite{Deglise}.  Roughly speaking, this says homotopy invariance of cohomology and orientability together imply the existence of Voevodsky transfers.  Roughly speaking, our result says (assuming homotopy invariance for the presheaf itself) orientability and weak transfers together imply the homotopy invariance of cohomology.  There is an important conceptual difference between these results and ours: in the more structural approaches of Morel and D{\'e}glise, the Nisnevich topology is built in from the beginning, whereas we work with presheaves, so that proving Nisnevich excision is one of the main challenges.

Incidentally, the result of D{\'e}glise also suggests that the Nisnevich sheafification of a homotopy invariant presheaf with oriented weak transfers has Voevodsky transfers.  Additional evidence is the following pair of results on $K_0$-presheaves: the Zariski separation of a homotopy invariant $K_0$-presheaf is again a homotopy invariant $K_0$-presheaf \cite[Prop.~5.20]{WalkerThesis}, and a Zariski separated $K_0$-presheaf is a pretheory \cite[Thm.~6.23]{WalkerThesis}.

Finally, we also wish to mention two recent expansions of Voevodsky's theory in other directions.
The theory of reciprocity sheaves aims to extend constructions in the style of Voevodsky to capture non-homotopy invariant phenomenona \cite{KRec}, \cite{IR}, \cite{KSY}.  
Kelly has defined presheaves with traces and the $\ell dh$ topology as a framework for motives in positive characteristic, using alterations of singular varieties as a replacement for resolution of singularities \cite{Kelly}.
By contrast, our methods rely heavily on the hypothesis of homotopy invariance, and our goal is to work entirely with smooth correspondences.

\section{Presheaves with weak transfers on semilocal schemes}

\subsection{Geometric Presentation Lemma}

The following lemma enhances the presentation lemma of Panin-Zainoulline \cite[3.5]{PanZBO} (which itself enhances the presentation lemma of Panin-Ojanguren \cite[Sect.~10]{PO}) by controlling the singularities of the presentation at $0$ and $1$.  
 
\begin{lemma}  \label{geom pres} Let $k$ be a perfect infinite field.  Let $R$ be a semilocal essentially smooth $k$-algebra, and $A$ an essentially smooth $k$-algebra which is finite over $R[t]$.  Suppose given an $R$-augmentation $\epsilon : A \to R$ and that $A$ is $R$-smooth at every prime containing $I:=\ker \epsilon$.  Finally suppose given $f \in A$ such that $A / f A$ is $R$-finite.  Then there exists $s \in A$ such that:
\begin{enumerate}
\item $A$ is $R[s]$-finite;
\item $A / sA = A / I \times A / J$ for some ideal $J \subset A$, with $A/J$ essentially smooth and $R \to A/J$ generically \'{e}tale;
\item $J + fA = A$; and
\item $(s-1)A + fA = A$, with $A / (s-1)A$ essentially smooth and $R \cong R[s] / (s-1) \to A / (s-1)A$ generically \'{e}tale.
\end{enumerate}
\end{lemma}

\begin{proof} One finds an element $\alpha \in I \subset A$ with suitable vanishing properties.  The main point is that for almost all $r \in k^\times$, the element $s := \alpha - rt^N$ (here $N$ is large compared to the degrees of the coefficients $p_i(t)$ of an equation expressing the integral dependence of $\alpha$ on $R[t]$; see the proof in \cite{PO} for further explanation of the notation) satisfies the conclusions.  Then for almost all $r \in k^\times$, the element $s$ will have the required smoothness and \'etaleness properties.

The ideal $I$ is locally principal since it is a section to smooth morphism of relative dimension 1\comment{CHECK QUILLEN reference SGA/EGA}.  Locally on $\spec A$, we can write $\alpha = a \alpha^\prime, t = a^l t^\prime$, with $(a) = I$ and $\alpha^\prime, t^\prime \notin I$.  By factoring out a suitable power of $a$ we obtain a local defining equation for the residual part $\spec (A/ J)$ of $V(s)$.  By Bertini's theorem, $V(s)$ is regular for most choices of $r \in k^\times$.  Since there is a finite cover of $\spec A$ on which $I$ is principal, we can find an $r$ which works uniformly.  Similarly, most choices of $r$ will produce a regular fiber at $s=1$.  The purity of the branch locus, together with the assumed $R$-smoothness along $A/I$, implies the resulting finite extensions are generically \'etale.
\end{proof}

\subsection{Semilocal vanishing is detected generically}

For a scheme $X$ and a finite set $S \subset X$, let ${\cal{O}}_{X,S}$ denote the semilocalization of $X$ at $S$.  For an irreducible variety $X$, let $\eta_X \in X$ denote its generic point.  
Since $\eta_X$, for example, is not generally of finite type over $k$, we define $\cal{F}(\eta_X)$ as the limit of the values of $\cal{F}$ on all (finite type) open subschemes of $X$.  For any open $V \subset X$, let $j_V :V \to X$ denote the open immersion.  In this section we generalize results from \cite[Sect.~4]{CTPST}.

\begin{theorem} \label{semilocal efface} Let $k$ be a perfect infinite field.  Suppose $\cal{F}$ is a homotopy invariant presheaf of abelian groups on ${\bf{Sm}} /k$ with weak transfers for affine varieties.  Let $X$ be an irreducible smooth affine $k$-scheme, $\dim X = d$, such that $\Omega^1_{X/k} \cong \cal{O}^d_X$.  Let $S \subset X$ be a finite set, and $Z \hookrightarrow X$ a closed subset.  Then there exists a neighborhood $U$ of $S$ and a map $a: \cal{F}(X \setminus Z) \to \cal{F}(U)$ such that the following diagram commutes.
$$\xym{ \cal{F}(X) \ar[r]^-{j_{X \setminus Z}^*} \ar[d]_-{j_U^*} & \cal{F}( X \setminus Z) \ar[ld]^-a \\ \cal{F}(U) }$$
 \end{theorem}
 
 \begin{corollary} \label{semilocal inj} For $k$ and $\cal{F}$ as in Theorem \ref{semilocal efface}, and $U$ a nonempty open subscheme of a smooth semilocal $k$-scheme $S$, the map $\cal{F}(S) \to \cal{F}(U)$ is a split monomorphism.  In particular the map $\cal{F}(\spec ({\cal{O}}_{X,S})) \to \cal{F}(\eta_X)$ is a split monomorphism. \end{corollary}
 \begin{proof} We can find a neighborhood of $S$ such that $\Omega^1_{X/k} \cong \cal{O}^d_X$. \end{proof}

\begin{corollary} \label{MVW 11.2} For $k$ and $\cal{F}$ as in Theorem \ref{semilocal efface}, if $\cal{F}( \spec E) =0$ for all fields $E \supset k$, then $\cal{F}_{Zar} = 0$. \end{corollary}

\begin{corollary} \label{V 4.19} \label{MVW 22.8} For $k$ and $\cal{F}$ as in Theorem $\ref{semilocal efface}$, and $U$ a nonempty open subscheme of a smooth $k$-scheme $X$, the restriction map $\cal{F}_{Zar}(X) \to \cal{F}_{Zar}(U)$ is a monomorphism. \end{corollary}

\begin{corollary} \label{V 4.20} Let $k$ be a perfect infinite field, and let $\varphi : \cal{F} \to \cal{G}$ be a morphism of homotopy invariant presheaves of abelian groups on $\Sm / k$ with weak transfers for affine varieties.
Suppose that for any field extension $E \supset k$ the morphism $\varphi_E : \cal{F} (\spec E) \to \cal{G} (\spec E)$ is an isomorphism.  Then the morphism of associated sheaves $\varphi_{Zar} : \cal{F}_{Zar} \to \cal{G}_{Zar}$ is an isomorphism.  \end{corollary}
\begin{proof} The presheaves $\ker \varphi$ and $\coker \varphi$ inherit structures of homotopy invariant presheaves with weak transfers for affine varieties, so Theorem $\ref{semilocal efface}$ applies to them.  Since sheafification is exact it suffices to show the Zariski sheafification of the kernel and cokernel both vanish.  By hypothesis we know they vanish on fields, hence their stalks vanish by Theorem $\ref{semilocal efface}$. \end{proof}

\begin{remark} These results were obtained for $K_0$-presheaves by Mark Walker \cite[5.28, 5.30]{WalkerThesis}.
\end{remark}

The proof of the following result is postponed because the methods are slightly more involved (see $\ref{semilocal MV2}$), in particular the orientation is involved.  
 
 \begin{corollary} Let $k$ be a perfect infinite field.  Suppose $\cal{F}$ is a homotopy invariant presheaf of abelian groups on ${\bf{Sm}} /k$ with oriented weak transfers for affine varieties.
Let $S$ be a semilocal smooth $k$-variety and $S = U_0 \cup V$ a Zariski cover.
Then there exists an open $U \subset U_0$ such that $S=U \cup V$ and the following sequence is exact:
 $$0 \to \cal{F}(X) \to \cal{F}(U) \oplus \cal{F}(V) \to \cal{F}( U \cap V) \to 0 .$$  \end{corollary}

\begin{proof}[Proof of Theorem $\ref{semilocal efface}$] By Quillen's trick \cite{Q}, we can find a finite surjective morphism $p : X \to \mathbb{A}^d$ such that composing with the projection away from the (say) first factor gives a morphism $X \to \mathbb{A}^{d -1}$ of relative dimension 1, smooth at $S$, and so that $Z$ is finite over $\mathbb{A}^{d-1}$.  For any open $U \subset X$ (which we will assume contains $S$), by applying $- \times_{\mathbb{A}^{d-1}} U$, we obtain the following diagram:
$$\xym{ Z \times_{\mathbb{A}^{d-1}} U \ar[r]^-{\text{cl.imm.}} \ar[rd]_-{\text{finite}} & X \times_{\mathbb{A}^{d-1}} U \ar[d]  \ar[r] & \A \times U \cong \mathbb{A}^d \times_{\mathbb{A}^{d-1}} U \ar[ld] \\ & U \ar@/_1pc/[u]_-\Delta }$$

We have an exact sequence of sheaves $0 \to p^\ast(\Omega^1_{\mathbb{A}^{d-1}/k}) \to \Omega^1_{X/k} \to \Omega^1_{ X / \mathbb{A}^{d-1}} \to 0$, all of which are locally free near $S$.  Since $\Omega^1_{\mathbb{A}^{d-1}/k} \cong {\cal{O}}^{d-1}_{\mathbb{A}^{d-1}}$,  taking determinants shows $\Omega^1_{ X / \mathbb{A}^{d-1}} \cong \cal{O}_X$.  As $X$ is affine, the sequence also gives $\Omega^1_{X/k} \cong {\cal{O}}^d_X$.  (Hence also $\Omega^1_{U/k} \cong {\cal{O}}^d_U$.)
 
By the geometric presentation lemma (Lemma $\ref{geom pres}$), there is a morphism $X \times_{\mathbb{A}^{d-1}} U \to \mathbb{A}^1$ such that the induced morphism $\pi : X \times_{\mathbb{A}^{d-1}} U \to U \times \mathbb{A}^1$ of $U$-schemes has the following properties:

\begin{enumerate}
\item $\pi$ is finite (hence flat, since source and target are regular of the same dimension);
\item $\pi^{-1}(U \times 0) = \Delta(U) \coprod R_0$;
\item $\delta: \Delta(U) \to U$ \'{e}tale;
\item $R_0$ is regular (hence smooth), $R_0 \subset (X \setminus Z) \times_{\bb{A}^{d-1}} U $, and $r_0 : R_0 \to U$ is generically \'{e}tale; and
\item $\pi^{-1}(U \times 1) =: F_1$ is regular (hence smooth), $F_1 \subset (X \setminus Z) \times_{\bb{A}^{d-1}} U$, and $f_1: F_1 \to U$ generically \'{e}tale.
\end{enumerate}

We claim the morphism $\pi$ admits an $N$-trivial embedding.  By Lemma \ref{N trivial A1 stable} it suffices to show $X \times_{\mathbb{A}^{d-1}} U \to U$ admits an $N$-trivial embedding.  Since $U \to \bb{A}^{d-1}$ is smooth, by Lemma \ref{N trivial smooth BC}, it suffices to show $X \to \bb{A}^{d-1}$ admits an $N$-trivial embedding.  This is a consequence of Lemma \ref{N trivial no kahler}.  We choose such an embedding.  The geometric situation is summarized in the following diagram.

$$\xym{ X \times_{\mathbb{A}^{d-1}} U \ar[r]^-i \ar[rd]_-\pi  & X \times_{\mathbb{A}^{d-1}} U \times \A \ar[d]^-{{pr}_{23}} \\
& U \times \A }$$

Hence we have a commutative diagram in which the paired arrows are inverse isomorphisms, and the unlabeled arrows pointing down are induced by the base change $0 \into \A$:

$$\xym{
\cal{F}(X \times \A) \ar[r]^-{p^*_{13}} \ar@<+1ex>[d] & \cal{F}(X \times_{\mathbb{A}^{d-1}} U \times \A) \ar[r]^-{i^*} \ar@<+1ex>[d] & \cal{F}(X \times_{\mathbb{A}^{d-1}} U) \ar[d] \ar[r]^-{\pi_\ast} & \cal{F}(U \times \A) \ar@<+1ex>[d]  \\
\cal{F}(X) \ar[u]^-{{p^X_1}^*} \ar[r]_-{p^*_1} & \cal{F}(X \times_{\mathbb{A}^{d-1}} U) \ar[r]_-{i_0^*} \ar[u]^-{p^*_{12}}& \cal{F}(\Delta) \oplus \cal{F}(R_0)  \ar[r]_-{ \delta_\ast + {r_0}_\ast } & \cal{F}(U) \ar[u]^-{{p^U_1}^**} }$$

We omit the choices of $N$-trivial embedding and trivializations in the notation, with the understanding that
the embeddings are induced from an $N$-trivial embedding of $\pi$, and whichever trivialization is used to define $\pi_*$ is also used to define $\delta_*, {r_0}_*, {f_1}_*$.  Note that $\delta_\ast \circ {pr}_{\cal{F}(\Delta)} (i_0^*) \circ p^*_1 = j_U^* : \cal{F}(X) \to \cal{F}(U)$.

We can consider the similar diagram induced by $1 \into \A$, replacing $\cal{F}(\Delta) \oplus \cal{F}(R_0)$ with $\cal{F}(F_1)$.  Then we conclude $(\delta_\ast + {r_0}_\ast) \circ i_0^* \circ p^*_1 = {f_1}_\ast \circ i_1^* \circ p^*_1 : \cal{F}(X) \to \cal{F}(U)$, as both are equal to ${p^U_1}^** \circ \pi_\ast \circ i^* \circ p^*_{13} \circ {p^X_1}^*$.

Since $R_0$ and $F_1$ avoid $Z$, we have a commutative diagram (for $C = R_0$ or $F_1$, $c= r_0, f_1$):
$$\xym{
\cal{F}(X) \ar[r]^-{p^*_1} \ar[d]_-{j_{X \setminus Z}^*} & \cal{F}(X \times_{\mathbb{A}^{d-1}} U) \ar[r]^-{\cal{F}({i_C})} \ar[d]_-{(j_{X \setminus Z} \times 1)^*} & \cal{F}(C) \ar[r]^{c_\ast} & \cal{F}(U) \\
\cal{F}(X \setminus Z) \ar[r]_-{p^*_1} & \cal{F}(X \setminus Z \times_{\mathbb{A}^{d-1}} U) \ar[ru]_-{i_C^*} }$$

We conclude the map $({f_1}_\ast \circ i_1^* \circ p^*_1) - ({r_0}_\ast \circ {pr}_{\cal{F}(R_0)} (i_0^*) \circ p^*_1 : \cal{F}(X) \to \cal{F}(U)$ factors through $j_{X \setminus Z}^*$.  By our previous calculations, this map is exactly $j_U^*$.

Observe we only used transfers along the finite generically \'{e}tale morphism $\pi : X \times_{\bb{A}^{d-1}} U \to U \times \A$ and its base changes along $0, 1 \into \A$, which are again generically \'{e}tale by purity of the branch locus.  \end{proof}

\section{Constructing smooth correspondences} \label{smooth corr}
\subsection{The main idea} \label{main idea}

In this section we analyze the geometric constructions used by Voevodsky \cite{CTPST}.  A presheaf $\cal{F}$ has the structure of a \textit{pretheory} if given any smooth relative curve $X \to S$ with $X,S \in \textbf{Sm} / k$ and any relative zero-cycle $\cal{Z}$ on $X$, there is a transfer (wrong-way) homomorphism $\phi_\cal{Z} : \cal{F}(X) \to \cal{F}(S)$.  The collection of transfer homomorphisms is required to satisfy some natural conditions.  This is a generalization of the more widely used notion of presheaf with transfers; presheaves with transfers are automatically pretheories of homological type.

Let $c_0(X/S)$ denote the group of relative zero-cycles on $X$ over $S$, and $h_0(X/S)$ the quotient of $c_0(X/S)$ by the relation of rational equivalence.  If the pretheory $\cal{F}$ is in addition homotopy invariant, then the homomorphism $\phi_- : c_0 (X/S) \to \Hom (\cal{F}(X) , \cal{F}(S) )$ factors through $h_0(X/S)$.

The main tool used to construct elements in $h_0(X/S)$ and to show relations among them is the following theorem of Suslin-Voevodsky \cite[3.1]{SV}: suppose $S$ is affine and $X \to S$ as above is quasi-affine and admits a compactification $\oline{X} \to S$ to a proper morphism of relative dimension 1, such that  $\oline{X}$ is normal and $X_\infty:= \oline{X} - X$ admits an $S$-affine neighborhood.  Then $h_0(X/S)$ is isomorphic to the relative Picard group $\pic (\oline{X}, X_\infty)$, the group of isomorphism classes of pairs $(L, t)$, where $L$ is a line bundle on $\oline{X}$ and $t : L |_{X_\infty} \cong \cal{O}_{X_\infty}$ is a trivialization of the restriction.  The map $\pic (\oline{X}, X_\infty) \to h_0(X/S)$ is given by lifting the trivialization to a rational section, then taking the zero scheme.  The ambiguity in the choice of lift corresponds exactly to $\A$-homotopy.

To ensure the zero schemes are \textit{effective} relative divisors, we will need to lift elements of the relative Picard group to global regular sections.  We will also need to ensure the zero schemes of these sections are smooth and multiplicity-free; this is accomplished by applying a Bertini theorem and we are thus restricted to characteristic zero.  The explicit $\A$-homotopies constructed in Proposition $\ref{Zariski homotopies}$ give the Mayer-Vietoris sequence for open subschemes of $\A$ (Theorem $\ref{open affine line MV}$).  In the next section similar constructions will be used to show that whether a local section of the presheaf $\cal{F}$ extends through a closed subset is invariant under Nisnevich covers (Theorem $\ref{general Nis invariance}$).

In the end the results of this section are restricted to the case where $k$ is of characteristic zero, though several intermediate results do not require this assumption.

\subsection{Mayer-Vietoris for open subschemes of the affine line}
We establish the Mayer-Vietoris sequence for open subschemes of the affine line.
As a consequence we compute the cohomology of open subschemes of $\A$ with coefficients in homotopy invariant presheaves with oriented weak transfers.

\begin{proposition} \label{A1 field injectivity1} Let $k$ be a field.  Let $V \subset U \subset \bb{A}^1_k$ be nonempty open subschemes of the affine line over $k$.  On $U \times \bb{P}^1$ there exist a line bundle $L$ and global sections $s_0, s_1 \in H^0(U \times \bb{P}^1, L)$ with the following properties:
\begin{enumerate}
\item the zero scheme $Z(s_0)$ is the disjoint union $\Delta_U \coprod R_0$;
\item $R_0$ is smooth, $U$-finite, and contained in $U \times V$;
\item the zero scheme $Z(s_1)$ is smooth, $U$-finite, and contained in $U \times V$;
\item $s_0 |_{U \times (\bb{P}^1 - U)} = s_1 |_{U \times (\bb{P}^1 - U)}$; and
\item the intersection $Z(s_0) \cap Z(s_1)$ is transverse.
\end{enumerate}
\end{proposition}

\begin{proof} We use the standard coordinates $X_0, X_1$ on $\bb{P}^1$, so that $\infty \into \bb{P}^1$ is defined by the vanishing of $X_0$, and $X_1 \over X_0$ is the canonical coordinate on $\bb{A}^1_{X_0 \neq 0} \subset \bb{P}^1$.  Let $u$ denote the canonical coordinate on $U$.  Let $F$ denote a homogeneous form generating the ideal of $\bb{A}^1 - U \into \bb{P}^1$ with the reduced scheme structure, i.e., $F \in \Gamma(\bb{P}^1, \cal{O}_{\bb{P}^1}(\A - U))$.  Let $G$ denote a form generating the ideal of $U -V \into \bb{P}^1$ with the reduced structure, and let $s_\Delta = X_1 - uX_0$ denote the canonical section of the diagonal bundle $\cal{O}(\Delta_U)$ on $U \times \bb{P}^1$.  By rescaling we may assume $F(X_0, X_1) = X_1^{\deg F} + \ldots$ and $G(X_0, X_1) = X_1^{\deg G} + \ldots$.  We let $f,g$ denote the dehomogenizations of $F,G$.

Let $N = \deg(FG) +1$.  The zero scheme of the section $s_\Delta F(X)G(X) + X_0^N \in H^0(U \times \bb{P}^1, \cal{O}(N))$ has the following properties: it is disjoint from $\Delta_U$; it is disjoint from $U \times (\bb{P}^1 - V)$; it is the graph of a function $V \to \bb{A}^1$, hence it is smooth; and it is $U$-finite.

We set $s_0 := s_\Delta (s_\Delta FG + X_0^N) \in H^0(U \times \bb{P}^1, L)$.  Here $L = \cal{O}(\Delta_U) \otimes \cal{O}(N) \cong \cal{O}(N+1)$.  With this choice of $s_0$, properties (1) and (2) are satisfied.  We write $s_{R_0} := s_\Delta FG + X_0^N$.

Now we find sections $s_1 \in H^0(U \times \bb{P}^1, L)$ satisfying (3)-(5).  The method used here is explicit but does not readily generalize.  There exist $a(x), b(x) \in k[x]$ such that $a(x) f(x) + b(x) g(x) =1$ and with $\deg (b) < \deg(f)$ and $\deg(a) < \deg(g)$.  Homogenizing we have (for some $r>0$): $$A(X) F(X) s_\Delta + B(X) G(X) s_\Delta = s_\Delta \cdot X_0^{N-r}.$$

Now for any $q(x)$ of degree $\deg(g) +2$ such that $f(x)q(x) = x^N + \ldots$ (i.e., the leading coefficient is $1$), the section
$$s_1(Q) := B(X) G(X) s_\Delta X_0^r + F(X) Q(X) $$
has the following properties:
\begin{itemize}
\item along $X_0 = 0$, we have $s_1(Q) = X_1^N (= s_0 |_{U \times \infty})$;
\item along $F=0$, we have $s_1(Q) = s_\Delta X_0^N (= s_0 |_{U \times (\bb{A}^1 - U)} )$; and
\item along $G=0$, we have $s_1(Q) = F(X) Q(X)$.
\end{itemize}

The first two properties imply $s_1(Q) = s_0$ along $U \times (\bb{P}^1 -U)$.  The third implies $s_1(Q)$ generates along $U \times (U-V)$ for $q$ generic.  Now we show that having fixed $B,G,F$, we can choose $Q$ appropriately.

The first item shows $Z(s_1(Q) ) \cap (U \times \infty) = \emptyset$ for any $Q$.  Choosing $Q$ so that $Z(Q)$ is disjoint from $Z(GB)$, we see that $Z(s_1(Q)) \cap (U \times Z(GB)) \subset U \times (Z(F) \cap Z(GB))$.  But the relation $af+bg =1$ implies that $Z(F) \cap Z(GB) = \emptyset$, hence $Z(s_1(Q)) \cap (U \times Z(GB)) = \emptyset$ and $Z(s_1(Q))$ is the graph of the function $u = { f(x) q(x) \over b(x) g(x) } + x$.  Hence for general $Q(X)$, the zero scheme $Z(s_1(Q))$ is smooth and disjoint from $U \times Z(G) = U \times  (U-V)$.

We need to check $Q$ can be chosen so the intersections $Z(s_1(Q)) \cap \Delta_U$ and $Z(s_1(Q)) \cap R_0$, i.e., 
$$\{ u = {f(x) q(x) \over b(x) g(x)} +x \} \cap \{ u=x \} \text{, and} $$
$$\{ u = {f(x) q(x) \over b(x) g(x)} +x \} \cap \{ u= { 1 \over f(x) g(x)} +x \},$$
are transverse.  For general $q$ the polynomials $fq$ and $f^2q -b$ have no multiple roots.  (The zero scheme of $F$ is reduced.)
\end{proof}

\begin{remark} \label{first bertini} A more conceptual proof (valid only in characteristic zero) using methods which generalize goes as follows.  Let $s_0$ be as in the proof.  Consider the sheaf sequence:
$$0 \to \cal{I}_{U \times (\bb{P}^1 - V)} \otimes L \to L \to L |_{U \times (\bb{P}^1 - V)} \to 0 .$$

The bundle $\cal{I}_{U \times (\bb{P}^1 - V)} \otimes L$ induces the bundle $\cal{O}(1)$ on every fiber, hence its $R^1{pr_1}_*$ vanishes.  Since $U$ is affine, this implies the vanishing of $H^1$.  Therefore the sheaf sequence remains exact after applying $H^0(U \times \bb{P}^1, - )$.  (This also holds if we replace $\cal{I}_{U \times (\bb{P}^1 - V)}$ by $\cal{I}_{U \times (\bb{P}^1 - U)} $.)  The Picard group of $U \times (U-V)$ is trivial.  Then for any generator $g \in H^0(U \times (U-V), L |_{U \times (U-V)})$, the element $(s_0 |_{U \times (\bb{P}^1 -U)}, g) \in H^0(U \times (\bb{P}^1 - V), L |_{U \times (\bb{P}^1 - V)} )$ lifts to a section $\widetilde{s} \in H^0(U \times \bb{P}^1, L)$.

Choose such a generator $g$ and a lift $\widetilde{s}$.  Consider the pencil $\PP_{s_0, \widetilde{s}}$ determined by $s_0$ and $\widetilde{s}$.   First we observe this pencil does not contain $\Delta_U$ as a fixed component: $Z(\widetilde{s}) \not \supset \Delta_U$ since $\widetilde{s}$ does not vanish along $U \times (U -V)$, while $s_\Delta |_{U \times (U-V)}$ vanishes along $(U-V) \times (U-V) \neq \emptyset$.  If the pencil $\PP_{s_0, \widetilde{s}}$ contains $R_0$ as a fixed component, then writing $\widetilde{s} = s_1 \cdot s_{R_0}$, the pair $s_0 = s_\Delta, s_1$ for $L = \cal{O}(\Delta_U)$ satisfies the conclusion of the proposition.  So we may assume the pencil contains neither $\Delta_U$ nor $R_0$ as a fixed component.  Since we have the decomposition $Z(s_0) = \Delta_U \coprod R_0$ into irreducible components, this implies the pencil does not contain any fixed component.

Every element $s_t \in \PP_{s_0, \widetilde{s}}$ of the pencil has the property that $s_t |_{U \times (\bb{P}^1 - U)} = s_0 |_{U \times (\bb{P}^1 - U)}$.  Since the element $\widetilde{s} \in \PP_{s_0, \widetilde{s}}$ generates along $U \times (U-V)$, the same is true of the general element of the pencil.  

Over a field of characteristic zero, the general element of a finite dimensional linear system on smooth quasi-projective variety is smooth away from the base locus of the linear system \cite[Cor.~10.9, Rmk.~10.9.2]{Hart}.  However $Z(s_0)$ is smooth everywhere, hence the general element of $\PP_{s_0, \widetilde{s}}$ is smooth everywhere, i.e., even along the base locus.  Note that it is enough to have a Bertini theorem over algebraically closed fields: if $k \into \oline{k}$ is an infinite subfield, the dense open locus of suitable sections in $\bb{P}^1$ contains many $k$-points essentially because $\bb{P}^1$ is a rational variety.  For details see \cite[Prop.~2.8]{HDbertini}.

Finally we explain how $\widetilde{s}$ can be modified to guarantee the intersection $Z(s_0) \cap Z(\widetilde{s})$ is transverse.  Let $e \in H^0 (\cal{I}_{U \times (\PP^1  - V)} \otimes L)$ be a section whose zero scheme is ``horizontal," i.e., independent of $u \in U$.  Then $Z(s_0) \cap Z(e)$ is transverse since the components of $Z(s_0)$ are graphs of functions $V \to U$.  Hence for a general $\lambda \in k$, the intersection $Z(s_0) \cap Z(\widetilde{s} + \lambda e)$ is also transverse.  Clearly $\widetilde{s}$ and $\widetilde{s} + \lambda e$ agree on $U \times (\PP^1 - V)$ for any $\lambda$, hence the arguments of the previous paragraphs apply, and $s_1 = \widetilde{s} + \lambda e$ satisfies the conclusion of the proposition.
\end{remark}

\begin{notation} \label{transfer mor} Let $Z,X,X' \in \Sm / k$, and suppose $Z \into X \times X'$ is a closed subscheme such that $Z \to X$ is finite flat and admits an $N$-trivial embedding.  Then we define $t_Z : \cal{F}(X') \to \cal{F}(X)$ to be the composition $\cal{F}(X') \to \cal{F}(Z) \to \cal{F}(X)$, where the first map is induced by the morphism $Z \into X \times X' \to X'$ and the second is the weak transfer along $Z \to X$. \end{notation}

\begin{corollary} \label{A1 field injectivity} Let $k$ be a field.  Let $\cal{F}$ be a homotopy invariant presheaf on $\Sm /k$ with weak transfers for affine varieties.  Let $V \subset U \subset \bb{A}^1_k$ be nonempty open subschemes of $\bb{A}_k^1$.  Then the restriction map $\cal{F}(U) \to \cal{F}(V)$ is injective. \end{corollary}

\begin{proof} The conditions guarantee the total space $\cal{Z} := Z(t \cdot s_0 + (1-t) \cdot s_1) \into \bb{A}^1_t \times U \times \bb{P}^1$ is $k$-smooth and $(\A_t \times U)$-finite.  The morphism $\cal{Z} \to \A_t \times U$ admits an $N$-trivial embedding.  Indeed the section $d(t \cdot s_0 + (1-t) \cdot s_1) \in H^0 (\cal{Z}, \Omega^1_{ \A_t \times U \times \PP^1} /k  |_\cal{Z} )$ trivializes the conormal bundle of the embedding $\cal{Z} \into \A_t \times U \times \PP^1$, which is canonically isomorphic to the conormal bundle of the embedding $\cal{Z} \into \A_t \times U \times \A_{X_0 \neq 0}$.  Hence the homotopy invariance gives the relation:
$$t_{\Delta_{U}} + t_{R_0} = t_{Z(s_0)} = t_{Z(s_1)} : \cal{F}(U) \to \cal{F}(U).$$
Since the correspondences $R_0$ and $Z(s_1)$ factor as $U \to V  \xrightarrow{j} U$, we get:
$$(t_{Z(s_1)} - t_{R_0} ) \circ j^* = t_{\Delta_{U}} = \id : \cal{F}(U) \to \cal{F}(U),$$
and so $j^*$ must be injective.   \end{proof}

\begin{proposition} \label{Zariski homotopies} Let $k$ be a field of characteristic $0$.  Let $ U \subset \bb{A}^1_k$ be an open subscheme, and $U = U_1 \cup U_2$ a Zariski cover.  Write $U_\infty := \bb{P}^1 - U$ and $Z_i := U - U_i$ for the pairwise disjoint reduced closed subschemes of $\bb{P}^1$.  On $U \times \bb{P}^1$ there exist:
\begin{enumerate}
\item a line bundle $M$; and
\item sections $\gamma \in H^0(U \times \bb{P}^1, M)$ and $s, s_1, s_2 \in H^0(U \times \bb{P}^1, \cal{O}(\Delta_U) \otimes M)$;
\end{enumerate}
such that the following equalities hold:
\begin{enumerate}
\item $s_1 = s_\Delta \cdot \gamma$ on $U \times (U_\infty \coprod Z_1)$;
\item $s_2 = s_\Delta \cdot \gamma$ on $U \times (U_\infty \coprod Z_2)$; and
\item $s_1 \cdot s_2 = s_\Delta  \cdot \gamma \cdot s$ on $U \times (U_\infty \coprod Z_1 \coprod Z_2)$.
\end{enumerate}

The sections also satisfy:
\begin{enumerate}
\item $\gamma$ generates along $\Delta_U$ and along $U \times (U_\infty \coprod Z_1 \coprod Z_2)$;
\item $s_1$ generates along $U \times Z_2$;
\item $s_2$ generates along $U \times Z_1$;
\item $Z(\gamma), Z(s_1), Z(s_2), Z(s)$ are $k$-smooth and $U$-finite; and
\item for $i=1,2$, the intersections $Z(s_i) \cap (\Delta_{U_{12}} \coprod Z(\gamma) )$ and $Z(s_i) \cap Z(s)$ are transverse in $U_{12} \times \bb{P}^1$.
\end{enumerate}

Finally there is an open subscheme $U_0 \subset U_{12}$ such that, letting ${(-)}_0$ denote the restriction, we have that ${Z(s_1)}_0 \cap {Z(s_2)}_0 = \emptyset$ and ${Z(s)}_0 \cap {(\Delta_{U} \coprod Z(\gamma))}_0 = \emptyset$.
\end{proposition}

\begin{remark} The conditions imply that $s_1$ generates along $U_1 \times (U_\infty \coprod Z_1)$, that $s_2$ generates along $U_2 \times (U_\infty \coprod Z_2)$, and that $s$ generates along $U \times (U_\infty \coprod Z_1 \coprod Z_2)$; and moreover that $s_2 = s$ on $U \times Z_1$ and $s_1 = s$ on $U \times Z_2$.
\end{remark}

\begin{remark} For the application (Theorem $\ref{open affine line MV}$), it would suffice to show the following weaker version of (5): for $i=1,2$, the intersections $Z(s_i) \cap Z(s)$ are transverse in $U_{12} \times \bb{P}^1$, possibly after ignoring irreducible components common to $Z(s_i)$ and $Z(s)$. \end{remark}

\begin{proof} Let $F$ be a homogeneous form defining $U_\infty$, i.e., a global section of $\cal{O}_{\bb{P}^1}(U_\infty)$; and let $G_1,G_2$ be forms defining $Z_1, Z_2$.  Let $u$ be the coordinate on $U$ and $X_0,X_1$ the coordinates on $\bb{P}^1$.  Then we set $\gamma := s_\Delta F(X) G_1(X) G_2(X) + {X_0}^n$, and $M = p_2^* (\cal{O}(n))$.  The section $\gamma$ generates along $\Delta_U$ and along $U \times (U_\infty \coprod Z_1 \coprod Z_2)$ as required.

Both $U \times Z_1$ and $U \times Z_2$ have trivial Picard group.  We choose also sufficiently general generating sections $g_1$ of $M(\Delta_U) := M \otimes \cal{O}(\Delta_U)$ along $U \times Z_2$ and $g_2$ along $U \times Z_1$.  Now we explain how to find $s_1$ as desired; $s_2$ is gotten by switching the indices.

The bundle $\cal{I}_{U \times (U_\infty  \coprod Z_1 \coprod Z_2)} \otimes M(\Delta_U)$ induces $\cal{O}(2)$ on every fiber of the first projection, and $U$ is affine, so the sheaf sequence:
$$0 \to \cal{I}_{U \times (U_\infty  \coprod Z_1 \coprod Z_2)} \otimes M(\Delta_U) \to M(\Delta_U) \to M(\Delta_U) |_{U \times (U_\infty  \coprod Z_1 \coprod Z_2)} \to 0$$
remains exact after applying $H^0(U \times \bb{P}^1, -)$.  So there are many global sections inducing $s_\Delta \cdot \gamma$ along $U \times (U_\infty \coprod Z_1)$ and $g_1$ along $U \times Z_2$.  The general such lift is smooth since the zero scheme of $s_\Delta \cdot \gamma$ is smooth and the condition that the section generates along $U \times Z_2$ is more general than the behavior of $s_\Delta \cdot \gamma$ along $U \times Z_2$.  More precisely one repeats the argument in Remark $\ref{first bertini}$ (possibly changing the generating section $g_1$).  This shows $Z(s_1)$ is $k$-smooth and intersects $\Delta \coprod Z(\gamma)$ transversely (in particular $Z(s_1)$ contains neither $\Delta$ nor $Z(\gamma)$ as an irreducible component).  

To produce the section $s$, choose a general global section $s^a$ of $M(\Delta_U)$ such that $s^a$ agrees with $s_\Delta \cdot \gamma$ along $U \times U_\infty$, with $s_2$ on $U \times Z_1$, and with $s_1$ on $U \times Z_2$.  We can find such a $s^a$ with smooth zero scheme since $Z(s_1)$ and $Z(s_2)$ are smooth.  Then choose a general global section $e$ of $\cal{I}_{U \times (U_\infty  \coprod Z_1 \coprod Z_2)} \otimes M(\Delta_U)$ with ``horizontal" zero scheme.

The sections $s_1$ and $s_2$ cannot have horizontal zero schemes, so the pencil determined by $s^a$ and $e$ contains no component of $Z(s_1) \cup Z(s_2)$ as a fixed component.  For $e$ sufficiently generic the pencil $\bb{P} (\lambda s^a + \mu e) \into U_1 \times \bb{P}^1 \times \bb{P}^1_{\lambda, \mu}$ will not have any base-points along $Z(s_1) \cup Z(s_2)$.   Furthermore all elements of the pencil (written as $s^a + {\mu \over \lambda} e$) have the correct behavior along $U \times (U_\infty \coprod Z_1 \coprod Z_2)$, so taking $s = s^a + {\mu_0 \over \lambda_0} e$, we can make the intersections transverse without disturbing the other properties.

To find the open subscheme $U_0 \subset U_{12}$, we simply discard the images of intersection points.  More precisely we set $C:= pr_1 (Z(s_1) \cap Z(s_2)) \cup pr_1 (Z(s) \cap (\Delta_{U_{12}} \coprod Z(\gamma) ) ) \into U_{12}$ and $U_0 := U_{12} - C$.  The subscheme $U_0$ is nonempty since $Z(s_1), Z(s_2)$ and $Z(s), \Delta_{U_{12}} \coprod Z(\gamma)$ have no components in common.   \end{proof}

\begin{theorem} \label{open affine line MV} \label{MVW 22.4} Let $k$ be a field of characteristic $0$.  Let $ U \subset \bb{A}^1_k$ be an open subscheme, and $U = U_1 \cup U_2$ a Zariski cover.  Write $U_{12} := U_1 \cap U_2$.

Let $\cal{F}$ be a homotopy invariant presheaf of abelian groups on $\Sm /k$ with oriented weak transfers for affine varieties.  Then the Mayer-Vietoris sequence:
$$0 \to \cal{F}(U) \xrightarrow{j_1^*, j_2^*} \cal{F}(U_1) \oplus \cal{F}(U_2) \xrightarrow{j_{12/1}^* - j_{12/2}^*}\cal{F}(U_{12}) \to 0 $$
is exact. \end{theorem}

\begin{proof} We have the open immersions (for $i=1,2$) : $j_i : U_i \to U$ and $j_{12/i} : U_{12} \to U_i$.  We have also $j_{12} : U_{12} \to U$ and $j_0 : U_0 \to U_{12}$.

By Proposition $\ref{Zariski homotopies}$ we obtain smooth correspondences $\gamma: U \to U_{12}, s_1 : U_1 \to U_{12}, s_2 : U_2 \to U_{12},$ and $s : U \to U_{12}$.  These induce morphisms between various groups $\cal{F}(-)$, e.g., we obtain a morphism $t_\gamma : \cal{F}(U_{12}) \to \cal{F}(\gamma) \to \cal{F}(U)$ as described in Notation $\ref{transfer mor}$.

After applying $\cal{F}$, these morphisms are subject to the following relations, where we have omitted $\cal{F}(-)$ to simplify the notation.  (The first equation, e.g., means $t_s \circ j_{12}^* = t_\gamma \circ j_{12}^* + \id : \cal{F}(U) \to \cal{F}(U)$.)

\begin{eqnarray} \label{line1} j_{12} \circ s =  j_{12} \circ \gamma + \id  :U \to U  \\ 
\label{line2} j_{12/1} \circ s_1 =  j_{12/1} \circ \gamma \circ j_1 +  \id_1 :U_1 \to U_1 \\
\label{line3} j_{12/2} \circ s_2  =  j_{12/2} \circ  \gamma \circ j_2 + \id_2 : U_2 \to U_2\\
\label{line4} (s_1 \circ  j_{12/1} + s_2 \circ j_{12/2}) \circ j_0 = (s \circ j_{12} + \gamma \circ j_{12} + \id_{12} ) \circ  j_0 : U_0 \to U_{12} \\
\label{line5} j_{12/1} \circ s_2  = j_{12/1} \circ s \circ j_2 : U_2 \to U_1 \\
\label{line6} j_{12/2} \circ s_1 = j_{12/2} \circ s \circ j_1 : U_1 \to U_2 \\ \end{eqnarray}

Now we define the contracting homotopy, where we have written $j_-$ for $j_-^*$ in the top row to be consistent with the rest of the notation:
$$\xymatrix{
0 \ar[r]  &\cal{F}(U) \ar[rr]^-{j_1, j_2} \ar[d] && \cal{F}(U_1) \oplus \cal{F}(U_2) \ar[rr]^-{j_{12/1} - j_{12/2}} \ar[d]  \ar[dll]^-{0 + B_2} && \cal{F}(U_{12}) \ar[r]  \ar[d]  \ar[dll]^-{A_1, C-A_2} & 0    \\
0 \ar[r]  &\cal{F}(U) \ar[rr] && \cal{F}(U_1) \oplus \cal{F}(U_2) \ar[rr] && \cal{F}(U_{12}) \ar[r] & 0 }$$
where $B_2 = j_{12/2} \circ (s - \gamma), A_i = \gamma \circ j_i - s_i,$ and $C = (\gamma - s ) \circ j_2$.

\medskip
We verify the homotopy relations (using the convention $dh-hd$).
\begin{itemize}
\item $j_2 \circ B_2 + j_1 \circ 0 \stackrel{\text{defn}}{=} j_2 \circ (j_{12/2} \circ (s - \gamma)) = j_{12} \circ (s - \gamma) \stackrel{\ref{line1}}{=} \id$

\medskip

\item $- ( j_{12/1} \circ A_1) \stackrel{\text{defn}}{=} - (j_{12/1} \circ (\gamma \circ j_1 - s_1) ) \stackrel{\ref{line2}}{=} \id_1$

\medskip

\item $B_2 \circ j_1 + j_{12/2} \circ A_1  \stackrel{\text{defn}}{=} j_{12/2} \circ (s - \gamma) \circ j_1 + j_{12/2} \circ (\gamma \circ j_1 - s_1)  \stackrel{\ref{line6}}{=} j_{12/2} \circ s_1 - j_{12/2} \circ \gamma \circ j_1 + j_{12/2} \circ \gamma \circ j_1 -   j_{12/2} \circ s_1 =0$

\medskip

\item $B_2 \circ j_2   +  j_{12/2} \circ (C- A_2) = j_{12/2} \circ (s - \gamma) \circ j_2 + j_{12/2} \circ (\gamma -s ) \circ j_2  + j_{12/2} \circ (s_2 - \gamma \circ j_2) = j_{12/2} \circ (s_2 - \gamma \circ j_2) \stackrel{\ref{line3}}{=} \id_2 $

\medskip

\item $0 \circ j_2 +  j_{12/1}\circ (C- A_2) =  j_{12/1} \circ (\gamma -s ) \circ j_2  + j_{12/1} \circ (s_2 - \gamma \circ j_2) = - j_{12/1} \circ s \circ j_2 + j_{12/1} \circ s_2 \stackrel{\ref{line5}}{=} 0$

\medskip

\item $A_1 \circ j_{12/1} - (C- A_2) \circ j_{12/2}  = \gamma \circ j_{12} - s_1 \circ j_{12/1} + (s- \gamma) \circ j_{12} + \gamma \circ j_{12} - s_2 \circ j_{12/2} = \gamma \circ j_{12} + s \circ j_{12} - s_1 \circ j_{12/1} - s_2 \circ j_{12/2} = \id_{12}$
\end{itemize}

The final equality holds upon precomposition with $j_0$, and on the presheaf this corresponds to composing with the injection $\cal{F}(U_{12}) \to \cal{F}(U_0)$.  Hence the equation holds by Corollary $\ref{A1 field injectivity}$.
\end{proof}

\begin{remark} \label{orientation necessary} Given a reduced but reducible divisor $D$ in $\bb{A}^n$ with smooth irreducible components $D_1, \ldots, D_r$, our strategy is to define $t_D$ as the sum $\sum_i t_{D_i}$, then verify relations among various $t_D$'s after removing the points of incidence.  This requires the weak transfers to be oriented, or at least ``oriented over all open subschemes of $\bb{A}^1$."

More precisely, the definition of each $t_{D_i}$ requires a trivialization of the normal bundle $N_{D_i}(\bb{A}^n)$.  If there are nontrivial incidences $D_{ij} := D_i \cap D_j$, these trivializations will not glue to a trivialization of $N_{D - \cup_{i,j}D_{ij}}(\bb{A}^n- \cup_{i,j}D_{ij})$.

Here is a simple example which captures our situation \'etale locally.  Consider $V(xy) = V(x) \cup V(y) \into \bb{A}^2$.  Then $dy$ trivializes the conormal bundle $N^\vee_{V(y)}(\bb{A}^2)$ and $dx$ trivializes the conormal bundle $N^\vee_{V(x)}(\bb{A}^2)$.  These do not glue to a trivialization of the conormal bundle $N^\vee_{V(xy) - (0,0)}(\bb{A}^2 - (0,0) )$.  Indeed $d(xy) = x dy + y dx$ trivializes $N^\vee_{V(xy) - (0,0)}(\bb{A}^2 - (0,0) )$, but restricts to $x dy$ on $V(y) - 0$ and to $y dx$ on $V(x)- 0$.

In this example, we require the weak transfers to be compatible with the open immersion $V(x) - 0 \cong \A - 0 \subset V(x) \cong \A$, and also that the trivialization $d x |_{V(x) - 0}$ induces the same weak transfer as the trivialization $y dx$.  Thus the proof of Theorem \ref{open affine line MV} requires the weak transfers to be oriented.  In general, discarding some incidence signals the orientation is necessary.
\end{remark}

\begin{lemma} \label{V 4.15 proof} Let $\cal{F}$ be a presheaf of abelian groups on $\Sm /k$, and suppose for any subscheme $U \subset \A_k$ and any Zariski covering $U = U_1 \cup U_2$ we have that the sequence:
$$ 0 \to \cal{F}(U) \to \cal{F}(U_1) \oplus \cal{F}(U_2) \to \cal{F}(U_1 \cap U_2) \to 0$$
is exact.  Then for all such $U$ we have $H^0_{Zar}(U, \cal{F}_{Zar} ) = \cal{F}(U)$ and $H^i_{Zar}(U, \cal{F}_{Zar}) = 0 $ for $i \neq 0$. \end{lemma}
\begin{proof} This isolates the implication shown in \cite[Thm.~4.15]{CTPST}. \end{proof}

\begin{corollary} \label{V 4.15}  \label{MVW 22.5} Let $k$ be a field of characteristic $0$, and let $\cal{F}$ be a homotopy invariant presheaf of abelian groups on $\Sm /k$ with oriented weak transfers for affine varieties.  Then for any open subscheme $U \subset \A_k$, we have $H^0_{Zar}(U, \cal{F}_{Zar} ) = \cal{F}(U)$ and $H^i_{Zar}(U, \cal{F}_{Zar}) = 0 $ for $i \neq 0$. \end{corollary}
\begin{proof} Combine Theorem $\ref{open affine line MV}$ and Lemma $\ref{V 4.15 proof}$.  \end{proof}

\begin{corollary} \label{V 4.16} Let $k$ be a field of characteristic $0$, and let $\cal{F}$ be a homotopy invariant presheaf of abelian groups on $\Sm /k$ with oriented weak transfers for affine varieties.  Then the morphism $\A_k \to \spec k$ induces an isomorphism $\cal{F}_{Zar}(\spec k) \cong \cal{F}_{Zar}(\A_k)$. \end{corollary}

\subsection{The affine line over a smooth local scheme}

We will also the need the injectivity of the restriction map for open subschemes of the affine line over a local base.

\begin{proposition} \label{A1 local injectivity} Let $k$ be a field of characteristic $0$, and let $S$ be a smooth local $k$-scheme.  Let $V \subset U \subset \bb{A}^1_S$ be nonempty open subschemes with $0_S \into V$, and suppose $U - 0_S$ is affine.  On $U \times_S \bb{P}^1_S$ there are a line bundle $L$ and global sections $s_0, s_1 \in H^0(U \times_S \bb{P}_S^1, L)$ with the following properties:
\begin{enumerate}
\item the zero scheme $Z(s_0 |_{(U -0_S) \times_S \bb{P}^1_S})$ is the disjoint union $\Delta_{U-0_S} \coprod R_0$;
\item $R_0$ is $k$-smooth, $(U-0_S)$-finite, and contained in $(U-0_S) \times_S (V-0_S)$;
\item the zero scheme $Z(s_1)$ is $k$-smooth, $(U-0_S)$-finite, and contained in $(U-0_S) \times_S (V-0_S)$;
\item $s_0 |_{(U-0_S) \times_S (0_s \coprod ( \bb{P}_S^1 - U))} = s_1 |_{(U-0_S) \times_S ( 0_S \coprod (\bb{P}_S^1 - U))}$; and
\item the intersection $Z(s_0) \cap Z(s_1)$ is transverse.
\end{enumerate}
\end{proposition}

\begin{proof} We follow the proof of Proposition $\ref{A1 field injectivity}$ in Remark $\ref{first bertini}$: since $U - 0_S$ is affine, if a coherent sheaf $M$ on $(U - 0_S) \times_S \bb{P}^1_S$ satisfies $R^1 {pr_1}_*M =0$, then $H^1( (U - 0_S) \times_S \bb{P}^1_S, M) = 0$. Less significantly, $0_S$ plays the role of $\infty$.

We use the standard coordinates $X_0, X_1$ on $\bb{P}_S^1$, so that $0_S \into \bb{P}_S^1$ is defined by the vanishing of $X_1$, and $X_1 \over X_0$ is the canonical coordinate on ${(\bb{A}_S^1)}_{X_0 \neq 0} \subset \bb{P}_S^1$.  Let $u$ denote the canonical coordinate on $U$.  We have $\pic(U \times_S \bb{P}^1_S) \cong \bb{Z}$ via the degree in the second factor.  Let $F$ denote a homogeneous form (i.e., section of some $\cal{O}(a)$ on $\bb{P}^1_S$) such that $Z(F) \into \bb{P}^1_S$ satisfies $Z(F) \supset \bb{P}_S^1 - U$, and let $G$ be such that $Z(G) \supset U-V$.  Note we can achieve that $Z(FG) \cap 0_S  = \emptyset$.  Let also $s_\Delta = X_1 - uX_0$ denote the canonical section of the diagonal bundle $\cal{O}(\Delta_U)$ on $U \times_S \bb{P}_S^1$.

Now we let $\gamma := s_\Delta FG + X_1^N  \in H^0(U \times_S \bb{P}^1_S, pr_2^*\cal{O}(N))$ and $s_0 : = s_\Delta \cdot \gamma  \in H^0(U \times_S \bb{P}^1_S, \cal{O}(\Delta_U) \otimes pr_2^*\cal{O}(N))$.  Of course $R_0 = Z(\gamma)$ and $L = \cal{O}(\Delta_U) \otimes pr_2^*\cal{O}(N) \cong pr_2^* \cal{O}(N+1)$.  To verify the properties of  $Z(s_0 |_{(U -0_S) \times_S \bb{P}^1_S})$ we observe the following:
\begin{itemize}
\item $\gamma |_{U \times_S Z(FG)} = X_1^N  |_{U \times_S Z(FG)} \neq 0$;
\item $\gamma |_{\Delta_U}$ has a zero only along $0_S \times_S 0_S$, hence $\gamma |_{ \Delta_{U-0_S} }$ is nowhere zero; and
\item the restriction $\gamma |_{U_{\kappa(S)} \times_{\kappa(S)} \bb{P}^1_{\kappa(S)}}$ is the graph of a rational map $\bb{P}^1_{\kappa(S)} \to U_{\kappa(S)}$, hence the closed fiber ${(R_0)}_{\kappa(S)}$ is $k$-smooth, hence $R_0$ is $k$-smooth.
\end{itemize}

Consider the exact sheaf sequence:
$$ 0 \to \cal{I}_{(U-0_S) \times_S (0_S \coprod Z(FG))} \otimes L \to L \to L |_{(U-0_S) \times_S (0_S \coprod Z(FG))} \to 0.$$
\noi The bundle $\cal{I}_{(U-0_S) \times_S (0_S \coprod Z(FG))} \otimes L$ induces $\cal{O}(1)$ on every fiber, hence has no $R^1 {pr_1}_*$.  Since $U-0_S$ is affine, this bundle has no $H^1$ and the sequence remains exact after applying $H^0(-)$.  The Picard group of $(U - 0_S) \times_S Z(G)$ is trivial.  Therefore any element of $H^0(L |_{(U-0_S) \times_S (0_S \coprod Z(FG))})$ given by a pair of the form:
$$ \text{$(s_0$ along $(U-0_S) \times_S (0_S \coprod Z(F) )$, some generator $g$ along $(U-0_S) \times_S  Z(G)),$} $$
\noi can be lifted to a global section of $L$.  Since $Z(s_0)$ is $k$-smooth, for a general lift $s_1$ of $(s_0, g)$ for a general $g$, we will have $Z(s_1)$ is $k$-smooth and intersects $Z(s_0)$ transversely.

To see $R_0=Z(\gamma)$ and $Z(s_1)$ are contained in $(U-0_S) \times_S (V-0_S)$, we notice both $\gamma$ and $s_1$ were chosen to generate $L$ along $(U - 0_S) \times_S (0_S \coprod Z(FG) ) \supset (U - 0_S) \times_S (0_S \coprod (\bb{P}^1_S - V) )$. \end{proof}

\begin{corollary} Let $k$ be a field of characteristic $0$.  Let $\cal{F}$ be a homotopy invariant presheaf on $\Sm /k$ with weak transfers for affine varieties.  Let $S$ be a smooth local $k$-scheme, and let $0_S \to V \subset U \subset \bb{A}^1_S$ be open subschemes of $\bb{A}_S^1$.  Then the restriction map $\cal{F}(U - 0_S) \to \cal{F}(V-0_S)$ is injective. \end{corollary}

\begin{proof} The conditions guarantee the total space $Z(t \cdot s_0 + (1-t) \cdot s_1) \into \bb{A}^1_t \times_k (U-0_S) \times_S \bb{P}^1_S$ is $k$-smooth and finite over $\bb{A}^1_t \times_k (U-0_S)$, with trivial normal bundle.  Hence the homotopy invariance gives the relation:
$$t_{\Delta_{U - 0_S}} + t_{R_0} = t_{Z(s_0)} = t_{Z(s_1)} : \cal{F}(U - 0_S) \to \cal{F}(U - 0_S).$$
Since the correspondences $R_0, Z(s_1)$ factor as $U - 0_S \to V - 0_S \xrightarrow{j} U - 0_S$, we get:
$$(t_{Z(s_1)} - t_{R_0} ) \circ j^* = t_{\Delta_{U - 0_S}} = \id : \cal{F}(U - 0_S) \to \cal{F}(U - 0_S),$$
and so $j^*$ must be injective.   \end{proof}

The constructions made here are similar to those we will make in Proposition $\ref{Nis corr}$.  We employ the following notation.  On $\A_S \times_S \bb{P}^1_S$, $u$ is the canonical coordinate on the first factor and $X_0,X_1$ are homogeneous coordinates on the second factor.  In the second factor we have $0_S = Z(X_1), \infty_S = Z(X_0), 1_S = Z(X_1- X_0)$; also $s_\Delta = X_1 - u X_0$; and $0_S = V(u)$ in the first factor.  Finally we set $U_\infty := \bb{P}^1_S - U$.

\begin{proposition} \label{A1 local corr} Let $k$ be an infinite field, and let $S$ be a smooth local $k$-scheme.  Let $U \subset \bb{A}^1_S$ be an affine open neighborhood of $0_S \into \A_S$.  On $\A_S \times_S \bb{P}^1_S$ there are a line bundle $L$ and global sections $s_0, s_1 \in H^0(\A_S \times_S \bb{P}_S^1, L), s_2 \in H^0(U \times_S \bb{P}_S^1, L)$ with the following properties:

\begin{enumerate}

\item $s_0$ and $s_1$ agree along $\A_S \times_S U_\infty$ and generate there;
\item along $\A_S \times_S 0_S$ we have ${s_0 \over s_1} = u$, the canonical coordinate on the first factor;
\item $s_1$ generates along $\A_S \times_S 0_S$;
\item $s_\Delta \cdot s_1 =  (X_1 -  X_0) \cdot s_0$ along $\A_S \times_S (\infty_S \coprod 0_S)$; both are sections of $L \otimes pr_2^* \cal{O}(1)$;

\item $s_\Delta \cdot s_2 = (X_1 - X_0) \cdot s_0$ along $U \times_S (U_\infty \coprod 0_S)$;

\item $s_2$ generates along $U \times_S (U_\infty \coprod 0_S)$;

\item the incidence $Z(s_1) \cap \Delta$ is disjoint from $0_S \times_S \bb{P}^1_S$;

\item the incidence $Z(s_0) \cap (\A_S \times 1_S) $ is disjoint from $0_S \times_S \bb{P}^1_S$;

\item there is a closed subset $Z_1 \into \A_S$ disjoint from $0_S$ such that:
\begin{enumerate}
\item $Z(t \cdot s_\Delta \cdot s_1 + (1-t) \cdot (X_1 -  X_0) \cdot s_0 ) \into \A_t \times (\A_S - Z_1 -0_S) \times_S (\A_S - 0_S)$ is smooth,
\item the fiber at $t=0$ is the disjoint union $((\A_S - Z_1 - 0_S) \times_S 1_S ) \coprod Z(s_0)$, and
\item the fiber at $t=1$ is the disjoint union $\Delta_{\A_S - Z_1 -0_S} \coprod Z(s_1)$;
\end{enumerate}

\item the incidence $Z(s_2) \cap \Delta$ is disjoint from $0_S \times_S \bb{P}^1_S$;
\item there is a closed subset $Z_2 \into U$ disjoint from $0_S$ such that:
\begin{enumerate}
\item $Z(t \cdot s_\Delta \cdot s_2 + (1-t)\cdot (X_1 -  X_0) \cdot s_0 ) \into \A_t \times (U - Z_2 -0_S) \times_S (U - 0_S)$ is smooth,
\item the fiber at $t=0$ is the disjoint union $((U - Z_2 - 0_S) \times_S 1_S ) \coprod Z(s_0)$, and
\item the fiber at $t=1$ is the disjoint union $\Delta_{U - Z_2 -0_S} \coprod Z(s_2)$; and
\end{enumerate}
\item via the first projections, the zero schemes are finite and admit $N$-trivial embeddings.
\end{enumerate}
\end{proposition}

The conditions (2) and (3) imply that  $s_0$ generates along $(\A_S - 0_S) \times_S 0_S$.  The conditions (4) and (5) imply $s_2 = s_1$ along $U \times_S (\infty_S \coprod 0_S)$.  With (2) this implies that ${s_0 \over s_2} = u$ along $U \times_S 0_S$.

\begin{proof} The complement of $U \subset \bb{P}^1_S$ is an effective Cartier divisor, so we write $Z(F) = U_\infty$ with $F \in \Gamma(\bb{P}^1_S, \cal{O}(a))$ for some $a \in \bb{Z}_{>0}$.  Note that $U_\infty \supset \infty_S$, so $X_0$ divides $F$.  Now we take $L = pr_2^* (\cal{O}(N))$ with $N=a+1$.  Additionally we choose a global section $A$ of $L |_{U \times \bb{P}^1_S}$ such that $A |_{U \times_S U_\infty} = {X_1^N(X_1 - X_0) \over s_\Delta} |_{U \times_S U_\infty}$, and such that $A$ vanishes along $U \times_S 0_S$.  (This is possible since $\cal{I}_{\A_S \times_S (U_\infty \coprod 0_S)} \otimes L$ is trivial on every fiber of the first projection, hence has no $H^1$.)  In particular $A$ generates along $U \times_S U_\infty$.  

We make the following observations.  The form $F$ is homogeneous in the variables $X_0, X_1$.
\begin{itemize}
\item The zero schemes of the sections $X_1^N$ and $(X_1 - X_0) \cdot F$ are disjoint on $\A_S \times_S \bb{P}^1_S$ since $Z(F) \cap 0_S = \emptyset$, so the total space of the pencil they determine is smooth.  Indeed even over the residue field, the general element of the pencil $\bb{P}(\lambda X_1^N + \mu (X_1 - X_0) \cdot \oline{F}) \into \bb{P}^1_{\kappa(S)} \times \bb{P}^1_{\lambda, \mu}$ is reduced.
\item The zero schemes of the sections $X_1^N$ and $s_\Delta \cdot F$ are disjoint on $(\A_S - 0_S) \times_S \bb{P}^1_S$.  (If $X_1 = 0$ then $F \neq 0$, so a common zero occurs at $X_1 = X_1 - u X_0 =0$, but $u$ is invertible so the equations imply $X_0 = X_1 = 0$.)  Hence the total space of this pencil is also smooth.  (In fact a general element of this pencil is the graph of a morphism even over the residue field, since $Z(X_1^N + s_\Delta \cdot F) \cap Z(X_0 \cdot F) = \emptyset$.)
\item Since $A$ generates along $U \times_S Z(F)$, the pencil determined by $A$ and $F \cdot X_0$ on $U \times_S \bb{P}^1_S$ is base-point free.  So the total space $\bb{P}(\lambda A + \mu F \cdot X_0 ) \into U \times_S \bb{P}^1_S \times \bb{P}^1_{\lambda, \mu}$ is smooth over $k$.
\end{itemize}

Therefore possibly after rescaling $F$ by an element from the ground field, we have that $s_1 := X_1^N + (X_1 - X_0) \cdot F$ and $s_0 := X_1^N + s_\Delta \cdot F \in H^0(\A_S \times_S \bb{P}^1_S, L)$, and also $s_2 := A + F \cdot X_0 \in H^0 (U \times_S \bb{P}^1_S, L)$, have smooth zero schemes.

Additionally we can achieve that the section $X_1^N + s_\Delta \cdot F$ is nonzero along $0_S \times_S 1_S$: if we write $F = \sum f_i X_0^i X_1^{a-i}$ with $f_i \in \Gamma(S, \cal{O}_S)$, then we have $X_1^N + s_\Delta \cdot F |_{0_S \times_S 1_S} = X_1^N ( 1 + \sum f_i)$.  So we require that $1 + \sum_i f_i \in \Gamma(S, \cal{O}_S)^\times$.  

There is one further open requirement on $s_2$: having chosen the sections $s_0, X_1 -X_0, s_\Delta$, their zero schemes are finite over $\A_S$, in particular finite over the local scheme $S \cong 0_S = V(u) \into \A_S$.  Therefore the scheme $Z(s_0) \cup Z(X_1 - X_0) \cup Z(s_\Delta)$ has finitely many closed points lying over $ V(u) \times_S \bb{P}^1_S$.  Now we require the section $s_2$ is nonzero at all of these points.  This can be achieved since $s_2$ was chosen from a base-point free pencil.  Thus (10) is satisfied.

It is now straightforward to check the equalities and generating behavior asserted in (1)-(6).  

To see (7), observe that $s_1 = s_\Delta = u =0$ implies $X_1 = 0$.  But then $s_1 |_{X_1 = 0} = - X_0 F$ has zero scheme disjoint from $X_1=0$.

To see (8), observe that $s_0 = X_1 - X_0 = u =0$ implies that $s_0 = X_0^N(1 + \sum f_i)$.  But one of our conditions is that $(1 + \sum f_i)$ is a unit, hence $s_0 =0$ implies $X_0 =0$, whence $X_1 =0$.

To see (9), first we remove the image in $\A_S$ of the incidences of (7) and (8); as it is the image of the closed incidence set via the proper morphism $\A_S \times_S \bb{P}^1_S$, this set is closed.  By the properties (7) and (8), the image of the incidence is disjoint from $0_S \into \A_S$.  This handles the singularities of the fibers at $t=0$ and $t=1$.  The singularities of the total space are supported on the base locus, so it suffices to show the base locus has non-dense image disjoint from $0_S$.  There are four pairs of components to check:
\begin{itemize}

\item $s_1 = s_0 = 0$.  This implies $u X_0 F = X_0 F$.  Now $s_1 |_{F =0} = X_1^N$, which has zero scheme disjoint from $\A_S \times_S Z(F)$.  Therefore $F \neq 0$.  Since $X_0$ divides $F$, we conclude $X_0 \neq 0$ as well.  Therefore the base locus arising from $Z(s_1) \cap Z(s_0)$ is supported over $u=1$.

\item $s_1 = X_1 - X_0 = 0$.  Since $s_1 |_{X_1 - X_0 = 0} = X_1^N$, this implies $X_1 =0$, whence $X_0 = 0$.

\item $s_{\Delta_{\A_S - 0_S}} = s_0 = 0$.  Since $s_0 |_{\Delta} = X_1^N$, this implies $X_1 =0$, whence $u X_0 = 0$.  Since $u$ is invertible, this means $X_0 = 0$.

\item $s_{\Delta_{\A_S - 0_S}} = X_1 - X_0 = 0$.  If $X_1 = X_0$ then neither can be zero.  Then $s_{\Delta_{\A_S - 0_S}} |_{X_1 -X_0 = 0} = (1-u)X_0$, so the base locus arising from $\Delta \cap ( \A_S \times_S 1_S)$ is supported over $u=1$.

\end{itemize}

Hence the total space is smooth if we remove $V(u-1) \into \A_S - 0_S$.  Then if we take $Z$ to be $V(u-1)$ together with the image of the incidences discussed above, the total space also decomposes at $t=0,1$ as required.

The set $Z_2$ in (11) is obtained similarly: we remove the image in $U$ of the incidences in (7) and (10).  By our choice of $s_2$ we know the incidences $Z(s_2) \cap Z(s_0)$ and $Z(s_2) \cap U \times_S 1_S$ are disjoint from $0_S \times_S \bb{P}^1_S$.  Then we take $Z_2$ to be the union of the images of these incidences in $U$.  By construction the total space of $Z(t \cdot s_0 \cdot (X_1 - X_0) + (1-t) \cdot s_\Delta \cdot s_2) \into \A_t \times (U - Z_2 - 0_S) \times_S (U - 0_S)$ is $k$-smooth and decomposes at $t=0,1$ as desired.

All of the zero schemes are contained in $\A_S \times_S \A_S$ or its affine open subscheme $U \times_S \A_S$, or the product of one of these schemes with the affine line over $k$.  These are smooth affine $k$-schemes with trivial sheaf of K\"ahler differentials, hence the conormal bundle of a smooth Cartier divisor is trivial.  Therefore the zero schemes admit $N$-trivial embeddings via the first projection; they are finite since the first projection factors as a closed immersion followed by the projection away from $\PP^1$.  This shows (12).
\end{proof}

\begin{proposition} \label{A1 local} Let $k$ be a field of characteristic $0$.  Let $\cal{F}$ be a homotopy invariant presheaf on $\Sm / k$ with oriented weak transfers.  Let $S$ be a smooth local $k$-scheme.  Let $U$ be an open affine neighborhood of $0_S \into \A_S$.  Then the canonical map $\cal{F}(\A_S - 0_S) / \cal{F} (\A_S) \to \cal{F}(U - 0_S) / \cal{F}(U)$ is an isomorphism. \end{proposition}

\begin{proof} We will define a map $\psi : \cal{F}(U - 0_S) \to \cal{F}(\A_S - 0_S)$ with the properties of the map $\psi$ constructed in Theorem $\ref{general Nis invariance}$.  By the first property $\psi$ descends to a map $\cal{F}(U - 0_S) / \cal{F}(U)  \to \cal{F}(\A_S - 0_S)$.  Then the second and third properties show $\psi$ induces the inverse to the natural map.

We use the line bundles and sections constructed in Proposition $\ref{A1 local corr}$ to construct $\psi$.  The sections $s_0$ and $s_1$ determine smooth correspondences $(\A_S - 0_S) \to (U - 0_S)$, hence morphisms $t_{Z(s_0)}, t_{Z(s_1)} : \cal{F}(U - 0_S) \to \cal{F}(\A - 0_S)$ on the presheaf $\cal{F}$.  Now we define
$$\psi := t_{Z(s_0)} - t_{Z(s_1)} : \cal{F}(U - 0_S) \to \cal{F}(\A - 0_S).$$

We have $Z(X_1 - X_0) = \A_S \times_S 1_S \into \A_S \times_S (U-0_S)$.  Therefore $\oline{t_{Z(X_1 - X_0)}} : \cal{F}(U - 0_S) \to \cal{F}(\A_S - 0_S) / \cal{F}(\A_S)$ is zero.  Similarly, considering where $s_1$ was required to generate, we have $Z(s_1) \into \A_S \times_S (U - 0_S)$, hence $\oline{t_{Z(s_1)}} : \cal{F}(U - 0_S) \to \cal{F}(\A_S - 0_S) / \cal{F}(\A_S)$ is zero.  Therefore $\oline{\psi} = \oline{t_{Z(s_0)}} : \cal{F}(U - 0_S) \to \cal{F}(\A - 0_S)/ \cal{F}(\A_S)$.

\textit{Property (1).}  We claim the composition
$$\cal{F}(U) \xrightarrow{j_{U - 0_S \subset U}^*} \cal{F}(U - 0_S) \xrightarrow{\psi} \cal{F}(\A - 0_S)$$
is zero.  For this it suffices to construct an $\A$-homotopy between $s_0$ and $s_1$ on $(\A - 0_S) \times_S U$.  Since $s_0  = s_1$ along $\A_S \times_S U_\infty$ and both sections generate there, we may use the zero scheme $Z(t \cdot s_0 + (1-t) \cdot s_1) \into \A_t \times (\A_S - V(u-1) - 0_S) \times_S U$.  We removed the base locus so the total space is guaranteed to be smooth.  Since the transfers are compatible with open immersions in the base, this homotopy shows that $\cal{F}(U) \to \cal{F}(U - 0_S) \xrightarrow{\psi} \cal{F}(\A_S - 0_S) \to \cal{F}(\A_S - V(u-1) - 0_S)$ is zero.  But $\cal{F}(\A_S - 0_S) \to \cal{F}(\A_S - V(u-1) - 0_S)$ is injective by Proposition $\ref{A1 local injectivity}$, so the claim follows.

\textit{Property (2).} We claim the composition
$$\cal{F}(\A_S - 0_S) / \cal{F}(\A_S) \xrightarrow{\oline{j_{U - 0_S \subset \A_S - 0_S}^* }} \cal{F}(U - 0_S) / \cal{F}(U) \xrightarrow{\oline{\psi}} \cal{F}(\A_S - 0_S) / \cal{F}(\A_S)$$
is the identity.  We use the total space $Z(t \cdot  s_\Delta \cdot s_1 + (1-t)\cdot (X_1 - X_0) \cdot s_0)$, regarded as a closed subscheme in $\A_t \times (\A_S - Z_1 - 0_S) \times_S (\A_S - 0_S)$.  This shows that, as maps $\cal{F}(\A_S - 0_S) \to \cal{F}(\A_S - Z_1 - 0_S)$,
$$0 = j_{\A_S - Z_1 - 0_S \subset \A_S - 0_S}^* \circ (t_{\Delta_{\A_S - 0_S}} +  (t_{Z(s_1)} - t_{Z(X_1 - X_0)} - t_{Z(s_0)} ) \circ  j_{U - 0_S \subset \A_S - 0_S}^*).$$

By Proposition $\ref{A1 local injectivity}$ we conclude that, as maps $\cal{F}(\A_S - 0_S) \to \cal{F}(\A_S - 0_S)$,
$$0 = t_{\Delta_{\A_S - 0_S}} + (t_{Z(s_1)} - t_{Z(X_1 - X_0)} - t_{Z(s_0)} ) \circ  j_{U - 0_S \subset \A_S - 0_S}^*.$$
Since $\oline{t_{Z(s_1)}} = \oline{t_{Z(X_1- X_0)}} = 0$ and $\oline{\psi} = \oline{t_{Z(s_0)}}$, we get 
$$0 = \oline{t_{\Delta_{\A_S - 0_S}}}  - ( \oline{\psi}  \circ j_{U - 0_S \subset \A_S - 0_S}^*) : \cal{F}(\A_S - 0_S) \to \cal{F}(\A_S - 0_S) / \cal{F}(\A_S)$$
and the claim follows.

\textit{Property (3).}  We claim the composition
$$\cal{F}(U - 0_S) / \cal{F}(U) \xrightarrow{\oline{\psi}} \cal{F}(\A_S - 0_S) / \cal{F}(\A_S)  \xrightarrow{\oline{j_{U - 0_S \subset \A_S - 0_S}^* }}  \cal{F}(U - 0_S) / \cal{F}(U) $$
is the identity.  The ideas are similar, though here we require a homotopy respecting all of $U_\infty$.  We use the total space of $Z(t \cdot s_\Delta \cdot s_2 + (1-t) \cdot (X_1 - X_0) \cdot s_0 ) \into \A_t \times (U - Z_2 - 0_S) \times_S (U - 0_S)$.  By looking at how the fibers at $t=0,1$ decompose, we conclude that, as maps $\cal{F}(U - 0_S) \to \cal{F}(U - Z_2 - 0_S)$,
$$0= j_{U - Z_2 - 0_S \subset U - 0_S}^* \circ (t_{\Delta_{U - 0_S}} + t_{Z(s_2)} + ( j^*_{U - 0_S \subset \A_S -0_S } \circ ( t_{Z(s_0)} - t_{Z(X_1 - X_0)})).$$
By Proposition $\ref{A1 local injectivity}$ we conclude that 
$$t_{\Delta_{U - 0_S}} + t_{Z(s_2)} = j^*_{U - 0_S \subset \A_S -0_S } \circ (t_{Z(s_0)} + t_{Z(X_1 - X_0)} ) : \cal{F} (U - 0_S) \to \cal{F}(U - 0_S).$$

We have $\oline{t_{Z(s_2)}} = \oline{t_{Z(X_1 - X_0)}} = 0 : \cal{F}(U-0_S) \to \cal{F}(U - 0_S) / \cal{F}(U)$ since the sections $s_2, X_1 - X_0$ extend to nonvanishing sections on $U \times_S (U_\infty \coprod 0_S)$, i.e., through $0_S$ in the base.  Combining this with $\oline{\psi} = \oline{t_{Z(s_0)}}$ and the fact that the transfers are compatible with the \'etale base change $U - 0_S \to \A_S - 0_S$, we find $$\oline{t_{\Delta_{U - 0_S}}}  = \oline{j_{U - 0_S \subset \A_S - 0_S}^*} \circ \oline{\psi} : \cal{F}(U - 0_S) \to \cal{F}(U - 0_S) / \cal{F}(U)$$
as desired.  \end{proof}

\begin{proposition} \label{V 4.11} Let $\cal{F}$ be a presheaf of abelian groups on $\Sm / k$ such that for any smooth local $k$-scheme $S$ and any open affine neighborhood $U$ of $0_S \into \A_S$, the natural map $\cal{F}(\A_S - 0_S) / \cal{F} (\A_S) \to \cal{F}(U - 0_S) / \cal{F}(U)$ is an isomorphism.  Then for any smooth $k$-scheme $Y$ there is a canonical isomorphism of sheaves 
$${(\cal{F}_{-1})}_{Zar} \cong \cal{F}_{(Y \times \A, Y \times 0 )}$$
on $Y_{Zar}$.
\end{proposition}
\begin{proof} The limit on the right hand side may be computed over all affine open neighborhoods of $0_S$.  Then this simply isolates the implication shown in \cite[Prop.~4.11]{CTPST}. \end{proof}

\begin{corollary} Let $k$ be a field of characteristic $0$.  Let $\cal{F}$ be a homotopy invariant presheaf on $\Sm / k$ with oriented weak transfers.  Then for any smooth $k$-scheme $Y$ there is a canonical isomorphism of sheaves 
$${(\cal{F}_{-1})}_{Zar} \cong \cal{F}_{(Y \times \A, Y \times 0 )}$$
on $Y_{Zar}$.
\end{corollary}
\begin{proof} Combine Proposition $\ref{A1 local}$ and Proposition $\ref{V 4.11}$. \end{proof}

\section{Nisnevich excision} \label{Nis section}
We will need the following excision statement for open subschemes of the affine line.  Let $U \subset \bb{A}^1_k$ be an open subscheme.  Then for any $a \in \cal{F}_{Nis}(U)$, and any closed point $x \in U$, there exists an open neighborhood $V \ni x$ such that the restriction of $a$ belongs to $\cal{F}(V) \subset \cal{F}_{Nis}(V)$.  More generally we will need an excision statement for a local smooth pair.  Both of these follow from the analogue of \cite[4.12]{CTPST}.  We also discuss applications to ``contractions" of presheaves.

\begin{theorem}\label{general Nis invariance}
Let $k$ be a perfect infinite field, and let $\cal{F}$ be a homotopy invariant presheaf on $\Sm / k$ with oriented weak transfers for affine varieties.
\noi Let $f : X_1 \to X_2$ be an \'{e}tale morphism in $\Sm / k$ with $\dim (X_1) = \dim (X_2) = n$, and let $Z_2 \into X_2$ be a smooth divisor such that $Z_1 := f^{-1}Z_2 \to Z_2$ is an isomorphism.  Then for any closed point $z \in Z_2$ there exists an open subscheme $z \in V \subset X_2$ and a morphism $\psi : \cal{F}(X_1 - Z_1) \to \cal{F}(V - V \cap Z_2)$ with the following properties.  

\begin{enumerate}
\item The composition $\cal{F}(X_1) \to \cal{F}(X_1 - Z_1) \xrightarrow{\psi} \cal{F}(V - V \cap Z_2)$ is zero.
\item The following diagram commutes:
$$\xym{
\cal{F}(X_2 - Z_2)  \ar[d]_-{f^*} \ar[rd] \\
\cal{F}(X_1 - Z_1) \ar[r]^-{\overline{\psi}} &  \cal{F}(V - V \cap Z_2) / \cal{F}(V) \\ }$$

\item The following diagram commutes:
$$\xym{
\cal{F}(X_1 - Z_1) \ar[rd] \ar[r]^-{\overline{\psi}} &  \cal{F}(V - V \cap Z_2) / \cal{F}(V) \ar[d]^-{\overline{f^*}} \\ 
&  \cal{F}(f^{-1}(V) - f^{-1}(V) \cap Z_1) / \cal{F}(f^{-1}(V)) \\ }$$
\end{enumerate}

(We use a bar to indicate an induced quotient map; unlabeled maps are restriction maps or maps induced by restriction maps.)
\end{theorem}

We begin with a lemma which realizes the situation in Theorem $\ref{general Nis invariance}$ as a morphism of relative curves.

\begin{lemma} \label{general Nis geom presentation} Let $k$ be an infinite field.
Let $f : X_1 \to X_2$ be an \'{e}tale morphism in $\Sm / k$ with $\dim (X_1) = \dim (X_2) = n$, and let $Z_2 \into X_2$ be a smooth divisor such that $Z_1 := f^{-1}Z_2 \to Z_2$ is an isomorphism.  Let $z \in Z_2$.  Then, possibly after shrinking $X_2$ about $z$, there exist:

\begin{enumerate}
\item a smooth $(n-1)$-dimensional variety $S$;
\item open immersions $X_i \subset \overline{X_i}$ for $i=1,2$;
\item morphisms $\overline{p_i} : \overline{X_i} \to S$ for $i=1,2$; and
\item an $S$-morphism $\overline{f} : \overline{X_1} \to \overline{X_2}$;
\end{enumerate}
satisfying the following properties:
\begin{enumerate}
\item $\overline{f}$ is finite and extends $f$;
\item the $\oline{X_i}$ are normal, $S$-projective curves;
\item set $X_{\infty,i} := \oline{X_i} - X_i$; then $Z_i \cup X_{\infty,i}$ has an $S$-affine neighborhood in $\oline{X_i}$ for $i=1,2$;
\item $Z_i \cap X_{\infty,i} = \emptyset$; and
\item $p_2 := \oline{p_2} |_{X_2}$ is smooth along ${p_2}^{-1}(p_2(z))$.
\end{enumerate}
\end{lemma}

\begin{proof}  This is an adaptation of Quillen's trick \cite[Lemma 5.12]{Q}.  We let $p : \bb{A}^n \to \bb{A}^{n-1}$ denote projection away from the first factor and $i_0 : \bb{A}^{n-1} \into \bb{A}^n$ the inclusion $(0, \text{id}) : \bb{A}^{n-1} \to \bb{A}^1 \times \bb{A}^{n-1} \cong \bb{A}^n$.  Possibly after replacing $X_2$ by a neighborhood of $z$, there exists a finite surjective morphism $\pi : X_2 \to \bb{A}^n$ with the following properties.

\begin{enumerate}
\item The morphism $\pi$ is \'{e}tale along $\pi^{-1} (\pi(z))$ (hence $p \circ \pi$ is smooth at $z$).
\item The morphism $p \circ \pi |_{Z_2} : Z_2 \to \bb{A}^{n-1}$ is finite and (by the previous property) \'{e}tale at $z$. 
\item In the cartesian diagram
$$\xym{
{(X_2)}_0 \ar[r] \ar[d] & X_2 \ar[d] \\
\bb{A}^{n-1} \ar[r]^{i_0} & \bb{A}^n \\ }$$
we have ${(X_2)}_0 = Z_2 \cup R_2$ with $z \not \in R_2$.  (The \'etaleness guarantees $Z_2$ is the unique branch through $z$.)  In particular $pr_1(Z_2) = 0 \in \bb{A}^1$.
\end{enumerate}
The morphism $X_2 \xrightarrow{\pi} \bb{A}^n \cong \bb{A}^1 \times \bb{A}^{n-1} \subset \bb{P}^1 \times \bb{A}^{n-1}$ is quasi-finite, hence by Zariski's Main Theorem we can find an open immersion $X_2 \subset \overline{X_2}$ and a finite morphism $\overline{\pi} : \overline{X_2} \to \bb{P}^1 \times \bb{A}^{n-1}$ which extends $\pi$.  We may replace an arbitrary $\oline{X_2}$ with its normalization.  The morphism ${pr}_2 \circ \overline{\pi} : \overline{X_2} \to \bb{A}^{n-1}$ is our candidate for $\oline{p_2} : \oline{X_2} \to S$.  The morphism $\oline{\pi}$ is finite, hence projective; and ${pr}_2 : \bb{P}^1 \times \bb{A}^{n-1}  \to \bb{A}^{n-1}$ is projective, hence $\oline{p_2}$ is projective.  

We have a closed immersion $X_{\infty,2} \into \overline{X_2}$, hence the induced morphism $X_{\infty,2} \to \bb{A}^{n-1}$ is proper.  Since also $X_{\infty,2} \subset {\overline{\pi}}^{-1}(\infty \times \bb{A}^{n-1})$, the morphism $X_{\infty,2} \to \bb{A}^{n-1}$ is quasi-finite.  Hence it is finite.

We claim $Z_2 \cup X_{\infty,2}$ admits an open neighborhood in $\overline{X_2}$ which is affine over $\bb{A}^{n-1}$.  We have $pr_1(Z_2) = 0 \in \bb{A}^1$ and $Z_2$ is irreducible, hence $pr_1(Z_2) = 0 \in \bb{A}^1$.  Then for any $t \in \bb{P}^1$ different from $0$ and $\infty$, the variety $\overline{X_2} - {\overline{\pi}}^{-1}(t \times \bb{A}^{n-1})$ is an affine open neighborhood of $Z_2 \cup X_{\infty,2}$.  Also $Z_2 \cap X_{\infty,2} = \emptyset$; this even holds after applying $pr_1$.

We claim $X_1$ can be compatibly compactified.  By applying Zariski's Main Theorem to the morphism $X_1 \to X_2 \subset \overline{X_2}$, we obtain an open immersion $X_1 \subset \overline{X_1}$ (with $\oline{X_1}$ normal) and a finite morphism $\overline{f} : \overline{X_1} \to \overline{X_2}$.  Therefore $\overline{X_1}$ is a projective curve over $S=\bb{A}^{n-1}$.  By the same arguments as above we conclude the subvariety $X_{\infty,1}$ is finite over $\bb{A}^{n-1}$, and $Z_1 \cup X_{\infty,1}$ has an open neighborhood in $\overline{X_1}$ which is affine over $\bb{A}^{n-1}$.  We have also ${\overline{f}}^{-1}(X_{\infty,2}) \subset X_{\infty,1}$, hence $\oline{f}$ restricts to a morphism $X_1 \to X_2$ which is nothing but our original $f$.
\end{proof}

A finite surjective morphism $f : X \to Y$ of normal varieties is flat away from a set of codimension at least 2.  Hence in this situation for any line bundle $L$ on $X$ we have the line bundle $\det f_* L$ on $Y$.  To a section $s$ of $L$ there is associated a section $N(s)$ of $\det f_* L$; if $s$ cuts out the Cartier divisor $D$ on $X$, then $N(s)$ cuts out $f_*D$ on $Y$.  If $M$ is a line bundle on $Y$, then $\det f_* f^*M \cong M^{\otimes \deg(f)}$. 

We apply these constructions to $V \times_S \oline{X_1} \to V \times_S \oline{X_2}$, which is the base change of $\oline{X_1} \to \oline{X_2}$ via the smooth morphism $V \times_S \oline{X_2} \to \oline{X_2}$.  This is permissible since the nonflat locus of $V \times_S \oline{X_1} \to V \times_S \oline{X_2}$ still has codimension at least 2.

To simplify the notations in the next proposition we use $\oline{f_2} := 1 \times \oline{f}$ and $f_1 := f  \times 1$.  We use the same general structure as in Proposition $\ref{A1 local corr}$.

\begin{proposition} \label{Nis corr} Let $k$ be a perfect infinite field.  Let the hypotheses and notation be as in Theorem $\ref{general Nis invariance}$, and suppose a compactification as in Lemma $\ref{general Nis geom presentation}$ has been chosen.  Let $d := \deg(\oline{f})$.  Then there exist:
\begin{enumerate}
\item an open neighborhood $V$ of $z_2$;

\item a line bundle $L$ on $V \times_S \oline{X_1}$;

\item sections $s_0, s_1 \in H^0(V \times_S \oline{X_1}, \oline{f_2}^* \cal{O}(\Delta_V) \otimes L )$;
\item a section $t_1 \in H^0 (V \times_S \oline{X_2}, \cal{O}((d-1) \Delta_V) \otimes \det \oline{f_2}_* L )$; and
\item a section $u_1 \in H^0 (f^{-1}(V) \times_S \oline{X_1}, \cal{O}(-\Delta_{f^{-1}(V)}) \otimes {f_1}^* (\oline{f_2}^* \cal{O}(\Delta_V) \otimes L)   )$;
\end{enumerate}

satisfying:

\begin{enumerate}
\item $s_0$ and $s_1$ agree on $V \times_S X_{\infty,1}$;

\item $s_1$ generates along $V \times_S ( X_{\infty,1} \coprod Z_1)$;

\item on $V \times_S  {(\oline{f})}^{-1} (X_{\infty,2}) \subset V \times_S X_{\infty,1}$, both $s_0$ and $s_1$ are given by $\oline{f_2}^*(s_\Delta) \cdot \ell$, where $\ell$ is a generating section of $L |_{V \times_S  {(\oline{f})}^{-1} (X_{\infty,2})}$;

\item via the isomorphism $V \times_S Z_1 \to V \times_S Z_2$, $s_0$ is given by $\oline{f_2}^* (s_\Delta) \cdot g$, where $g$ is a generating section of $L |_{V \times_S Z_1}$;

\item via $\oline{f_2} : V \times_S \oline{X_1} \to V \times_S \oline{X_2}$, the zero scheme $Z(s_0)$ maps isomorphically onto its image $Z(N(s_0))$;

\item $N(s_0)$ and $s_{\Delta_V} \cdot t_1$ agree on $V \times_S (X_{\infty,2} \coprod Z_2)$, and $t_1$ generates along $V \times_S (X_{\infty,2} \coprod Z_2)$;
\item ${f_1}^*(s_0)$ and $s_{\Delta_{f^{-1}(V)} } \cdot u_1$ agree on $f^{-1}(V) \times (X_{\infty,1} \coprod Z_1)$, and $u_1$ generates along $f^{-1}(V) \times_S (X_{\infty,1} \coprod Z_1)$;

\item the zero schemes $Z(s_0), Z(s_1), Z(t_1), Z(u_1)$ are $k$-smooth;

\item the zero schemes:
\begin{enumerate}
\item $Z(t \cdot s_0 + (1-t) \cdot s_1) \into \A_t \times (V - V \cap Z_2) \times_S (X_1 - Z_1)$,
\item $Z(t \cdot N(s_0) + (1-t) \cdot s_{\Delta_{V - V \cap Z_2}} \cdot t_1 )  \into \A_t \times (V - V \cap Z_2) \times_S (X_2 - Z_2)$, and
\item $Z(t \cdot f_1^* s_0 + (1-t) \cdot s_{\Delta_{f^{-1}(V) - f^{-1}(V \cap Z_2)}} \cdot u_1  )  \into \A_t \times (f^{-1}(V) - f^{-1}(V \cap Z_2)) \times_S (X_1 - Z_1)$
\end{enumerate}
are $k$-smooth; 
\item the zero schemes $Z(u_1), Z(t_1)$ are disjoint from the relevant diagonals:
\begin{enumerate}
\item ${\Delta_{V - V \cap Z_2}} \cap  Z(t_1) = \emptyset$ in $(V - V \cap Z_2) \times_S (X_2 - Z_2)$, and
\item ${\Delta_{f^{-1}(V) - f^{-1}(V \cap Z_2)}} \cap Z(u_1) = \emptyset $ in $(f^{-1}(V) - f^{-1}(V \cap Z_2)) \times_S (X_1 - Z_1)$; and
\end{enumerate}

\item via the first projections, the zero schemes are finite and admit $N$-trivial embeddings.
\end{enumerate}
\end{proposition}

\begin{remark}  The sections $N(s_0)$ and $s_{\Delta_V} \cdot t_1$  are compared via the isomorphism:

$$ \det  \oline{f_2}_* (  ( \oline{f_2}^*\cal{O}(\Delta_V)) \otimes L ) \cong \cal{O}(\Delta_V) \otimes \cal{O}( (d-1) \Delta_V ) \otimes \det  \oline{f_2}_* L.$$

The sections ${f_1}^*(s_0)$ and $s_{\Delta_{f^{-1}(V)} } \cdot u_1$ are compared via the isomorphism:

$$ {f_1}^* ( \oline{f_2}^*\cal{O}(\Delta_V)   \otimes L   ) \cong \cal{O}(\Delta_{f^{-1}(V)} ) \otimes \cal{O}(R) \otimes {f_1}^* L $$

where $R \into f^{-1}(V) \times \oline{X_1}$ is a smooth Cartier divisor disjoint from $\Delta_{f^{-1}(V)}$.

\end{remark}

\begin{proof} The zero schemes are automatically $V$-finite: we have $Z(s) \into V \times_S \oline{X}$ closed, and $V \times_S \oline{X} \to V$ is projective.  Hence the zero scheme is $V$-proper.
Note that since we can choose $V$ smooth over $S$, we have $V \times_S X_2$ is smooth over $X_2$, hence also $V \times_S X_2$ is $k$-smooth.

The choice of $L$ in (2) is dictated by the following criteria.  The first is that for any of the line bundles $M$ appearing in (3)-(5) (all of these involve $L$ and some twists) and any of the closed subschemes $Z$ among $V \times_S (X_{\infty1,} \coprod Z_1), V \times_S (X_{\infty,2} \coprod Z_2), f^{-1}(V) \times_S (X_{\infty,1} \coprod Z_1)$, it holds that $R^1 \oline{p}^* (\cal{I}_Z \otimes M) = 0$.  Replacing $S$ by an affine neighborhood of $p_2(z)$, the vanishing of $R^1p_*$ is equivalent to the standard restriction sequence remaining exact after applying $H^0(-)$.  The second criterion is that for all $Z,M$ as above, the linear subsystem $H^0(\cal{I}_Z \otimes M)$ of $H^0(M)$ is base-point free outside of $Z$.  Both can be achieved by choosing $L$ sufficiently ample relative to the morphism $V \times_S \oline{X_1} \to V$.  

Since we can shrink $V$ about $z$, we can ensure the zero schemes are smooth if they intersect the curves $z \times_S \oline{X_i}$ transversely.  We will deal carefully with the sections $s_0$ and $s_1$ and the curve $\oline{C}_z := z \times_S \oline{X_1}$; the others are similar.  Note $\oline{C}_z$ contains the smooth curve $z \times_S X_1$ as a dense open subscheme.

Let $\ell$ be a global section of $L$ which generates $L$ along $V \times_S {(\oline{f})}^{-1}(X_{\infty,2})$ and is not a $p^{\text{th}}$ power of a section (of a line bundle which is a $p^{\text{th}}$ root of $L$).  Let $g$ be a global section which generates along $V \times_S Z_1$.  Since $Z_1$ is finite over $S$, $Z_1$ is equal to its closure in $\oline{X_1}$.  Since $f$ is \'etale along $Z_1$, we have ${\oline{f}}^{-1}(Z_2) = Z_1 \coprod T$, where $T \subset X_{\infty,1}$.

We specify an element of $H^0((\oline{f_2}^* \cal{O}(\Delta_V) \otimes L) |_{V \times_S (X_{\infty,1} \coprod Z_1)})$, as in the statement, as follows:
\begin{itemize}
\item on $V \times_S {(\oline{f})}^{-1}(X_{\infty,2})$ the section restricts to $s_{{\oline{f_2}}^*\Delta} \cdot \ell$;
\item on $V \times_S ({(\oline{f})}^{-1}(Z_2) - Z_1)$ it generates; and
\item on $V \times_S Z_1$ it restricts to $s_{{\oline{f_2}}^*\Delta} \cdot g$.
\end{itemize}

By our assumption on $L$, we can find a global section $s^{init}_0$ of $\oline{f_2}^* \cal{O}(\Delta_V) \otimes L$ with this behavior on $V \times_S (X_{\infty,1} \coprod Z_1)$.  Now let $e_1 \in H^0(\cal{I}_{V \times_S ( {(\oline{f})}^{-1}(X_{\infty,2}) \coprod Z_1) } \otimes {\oline{f_2}}^* \cal{O}(\Delta_V) \otimes L)$ be a section which does not vanish along $V \times_S ({(\oline{f})}^{-1}(Z_2) - Z_1)$ and satisfies $Z(e_1) \cap Z(s_0^{init}) = V \times_S Z_1$.

Now we restrict the pencil determined by $s_0^{init}$ and $e_1$ to the curve $\oline{C}_z$, i.e., we consider the zero scheme $Z(\lambda s_0^{init} + \mu e_1 ) \into \oline{C}_z \times \bb{P}^1_{\lambda,\mu}$.  This pencil has the unique base point $z \times_S z$.  Hence $Z(\lambda s_0^{init} + \mu e_1)$ consists of two irreducible components: the graph of a morphism $\oline{C}_z \to \bb{P}^1$; and the component $(z \times_S z) \times \bb{P}^1$.  Since $k$ is perfect the extension $k \to \kappa(z)$ is separable and hence the morphism $(z \times_S z) \times \bb{P}^1 \to \bb{P}^1$ is separable.  By our choice of $\ell$ we know the morphism $\oline{C}_z \to \bb{P}^1$ induces a separable field extension over the point $\mu = 0$, hence $\oline{C}_z \to \bb{P}^1$  is separable.  Therefore $Z(\lambda s_0^{init} + \mu e_1) \to \bb{P}^1$ has smooth generic fiber, in other words the general element of the pencil $Z(\lambda s_0^{init} + \mu e_1 ) \into (V \times_S \oline{X_1}) \times \bb{P}^1$ intersects $\oline{C}_z$ transversely.  (In particular it avoids the singular locus of $\oline{C}_z$.)  We choose a general element $s_0(\lambda_0, \mu_0)$ of this pencil.

Now we can find $e_2 \in H^0(\cal{I}_{V \times_S X_{\infty,1}} \otimes {\oline{f_2}}^* \cal{O}(\Delta_V) \otimes L)$ such that $e_2$ does not vanish along $V \times_S Z_1$.  Since the subsystem to which $e_2$ belongs has no unassigned base locus, we can choose $e_2$ so that $Z(s_0(\lambda_0, \mu_0)) \cap Z(e_2) \cap \oline{C}_z = \emptyset$.  Now we let $s_0 := s_0(\lambda_0, \mu_0)$, and we let $s_1$ be a general element of the pencil determined by $s_0$ and $e_2$.  Since $s_1$ is more general than $s_0$, the zero scheme of $Z(s_1)$ also intersects $\oline{C}_z$ transversely.  Since $s_0=e_2=0$ has no solution along $\oline{C}_z$, the total space $Z(t \cdot s_0 + (1-t) \cdot s_1) \into \A_t \times \oline{C}_z$ is $k$-smooth, hence (possibly after shrinking $V$) the total space $Z( t s_0 + (1-t) s_1) \into \A_t \times (V \times_S \oline{X_1})$ is $k$-smooth.  (We just need: $Z( t \cdot s_0 + (1-t) \cdot s_1) \into \A_t \times (V - V \cap Z_2) \times_S (X_1 - Z_1)$ is $k$-smooth.)

Now we claim $Z(s_0)$ maps isomorphically onto its image $Z(N(s_0))$.  We can calculate the degree of the morphism $Z(s_0) \to Z(N(s_0))$ along the base change $V \times_S Z_2 \into V \times_S \overline{X_2}$.  We have ${\overline{f}}^{-1}(V \times_S Z_2) \subset (V \times_S (Z_1 \cup X_{\infty,1}))$ scheme-theoretically along $Z_1$; along $X_{\infty,1}$ there might be multiplicities.  Also $Z(s_0) \cap (V \times_S (Z_1 \cup X_{\infty,1})) = Z_2 \times_S Z_1$ scheme-theoretically, and $ Z_2 \times_S Z_1$ is degree 1 (indeed, an isomorphism) onto its image $Z_2 \times_S Z_2$.  Therefore $Z(s_0)$ itself maps with degree 1, i.e., isomorphically, onto $Z(N(s_0))$.  

Very similar arguments give the $k$-smoothness of $Z(t \cdot N(s_0) + (1-t) \cdot s_{\Delta_{V - V \cap Z_2}} \cdot t_1)$.  

After shrinking $V$ about $z$ if necessary we have $Z(N(s_0)) \cap \Delta_{V - V \cap Z_2} = \emptyset$.  Then an arbitrary $t_1$ with the prescribed behavior along $V \times (X_{\infty,2} \coprod Z_2)$ may not work, but we are free to add to any given $t_1$ a global section of $\cal{I}_{V \times_S (X_{\infty,2} \coprod Z_2) } \otimes \cal{O}((d-1)\Delta_V) \otimes \det \oline{f_2}_* L$.  Since this subsystem has no unassigned base locus, we can make $Z(t_1)$ disjoint from $Z(N(s_0))$ and $\Delta_V$ in the fiber curve $z \times_S \oline{X_2}$, hence this holds in a neighborhood of $z$ as well.

The situation with $f_1^*(s_0)$ and $u_1$ is slightly easier since we pull-back along an \'etale morphism rather than push-forward along a finite generically \'etale morphism.  The ideas are the same.

In all cases, we conclude the existence of $N$-trivial embeddings by shrinking $V$ so that $V$ and the smooth schemes finite over it have trivial sheaf of K\"ahler differentials.
\end{proof}

\begin{proof}[Proof of Theorem $\ref{general Nis invariance}$.]  Because the desired properties of $\psi$ are stable under precomposition with the restriction $\cal{F}(X_1 - Z_1) \to \cal{F}(f^{-1}(U) - f^{-1}(U) \cap Z_1 )$, we may replace $X_2$ by an open neighborhood $U$ of the point $z$, and $X_1$ by the preimage $f^{-1}U$.  We put the excisive morphism $f$ into the convenient relative form guaranteed by Lemma $\ref{general Nis geom presentation}$.  Now we define $\psi : \cal{F}(X_1 - Z_1) \to \cal{F}(V - V \cap Z_2)$ as the difference $\psi := t_{Z(s_0)} - t_{Z(s_1)}$, where $Z(s_0)$ and $Z(s_1)$ are the smooth correspondences constructed in Proposition $\ref{Nis corr}$, and $t_-$ is defined in Notation $\ref{transfer mor}$.

\textit{Property (1).} The correspondences $Z(s_0), Z(s_1) : (V - V \cap Z_2) \to (X_1 - Z_1) \subset X_1$ are $\bb{A}^1$-homotopic, i.e., they become homotopic when allowed to pass through $Z_1$.  To see this, we use the zero scheme $Z(t \cdot s_0 + (1-t) \cdot s_1 ) \into \A_t \times (V - V \cap Z_2) \times_S X_1$.  This is $k$-smooth by the properties in $\ref{Nis corr}$, and $Z(t \cdot s_0 + (1-t) \cdot s_1 ) \to \A_t \times (V - V \cap Z_2) $ admits an $N$-trivial embedding if $V$ is sufficiently small.

\textit{Property (2).} Considering where $s_1$ was required to generate, we see $Z(s_1) \into (V - V \cap Z_2) \times_S (X_1 - Z_1)$ extends through $V \cap Z_2$, i.e., to a smooth $V$-finite correspondence in $V \times_S (X_1 - Z_1)$.  In geometric terms we have a factorization

$$\xym{
V - V \cap Z_2 \ar[d] \ar[r]^-{Z(s_1)} & X_1 - Z_1 \\
V \ar[ru]_{{Z(s_1)}^-} \\}$$
where the arrow $V \to X_1 - Z_1$ is the (smooth) finite correspondence obtained by taking the closure of $Z(s_1)$.  Since the weak transfer along $Z(s_1)^-$ is compatible with the open immersion $V - V \cap Z_2 \subset V$, we conclude $\oline{t_{Z(s_1)}} : \cal{F}(X_1 - Z_1) \to \cal{F}(V - V \cap Z_2) / \cal{F}(V)$ is zero.  Hence we have

\begin{equation}  \label{prop2 eqn1} \overline{\psi} = \overline{t_{Z(s_0)}} : \cal{F}(X_1 - Z_1) \to \cal{F}(V - V \cap Z_2) / \cal{F}(V).  \end{equation}

Now the commutative diagram:
$$\xym{
Z(s_0) |_{V - V \cap Z_2} \ar[r] \ar[d]^-\cong & (V-V \cap Z_2) \times_S (X_1 - Z_1) \ar[r]^-{pr_2} \ar[d]^-{1 \times f} & X_1 - Z_1 \ar[d]^f \\
Z(N(s_0)) |_{V - V \cap Z_2} \ar[r] & (V-V \cap Z_2) \times_S (X_2 - Z_2) \ar[r]^-{pr_2}  & X_2 - Z_2 \\ }$$
shows that 

\begin{equation} \label{prop2 eqn2} t_{Z(s_0)} \circ f^* = t_{Z(N(s_0))} : \cal{F}(X_2 - Z_2) \to \cal{F}(V - V \cap Z_2) . \end{equation}

The homotopy $Z(t \cdot N(s_0) + (1-t) \cdot s_{\Delta_{V - V \cap Z_2}} \cdot t_1) \into \A_t \times ( V- V \cap Z_2) \times_S (X_2 - Z_2)$ shows that $t_{Z(N(s_0))} = t_{\Delta_{V- V \cap Z_2}} + t_{Z(t_1)} : \cal{F}(X_2 - Z_2) \to \cal{F}(V - V \cap Z_2)$.  Furthermore $Z(t_1)$ extends to a correspondence $V \to X_2 - Z_2$, so the map $\oline{t_{Z(t_1)}} : \cal{F}(X_2 - Z_2) \to \cal{F}(V - V \cap Z_2) / \cal{F}(V)$ is zero.

Therefore we have:
\begin{equation} \label{prop2 eqn3} \oline{t_{Z(N(s_0))}} = \oline{t_{\Delta_{V- V \cap Z_2}}} : \cal{F}(X_2 - Z_2) \to \cal{F}(V - V \cap Z_2) / \cal{F}(V). \end{equation}

Putting everything together, we conclude:

$$\oline{\psi} \circ f^* \stackrel{\ref{prop2 eqn1}}{=} \oline{t_{Z(s_0)}} \circ f^* \stackrel{\ref{prop2 eqn2}}{=} \oline{t_{Z(N(s_0))}} \stackrel{\ref{prop2 eqn3}}{=} \oline{t_{\Delta_{V - V \cap Z_2}}} :  \cal{F}(X_2 - Z_2) \to \cal{F}(V - V \cap Z_2) / \cal{F}(V)$$

as desired.

\textit{Property (3).} Our goal is to show $\oline{f^*} \circ \oline{\psi} : \cal{F}(X_1 - Z_1) \to \cal{F}(f^{-1}(V) - f^{-1}(V \cap Z_2)) / \cal{F}(f^{-1}(V))$ is the restriction map.  Since $\oline{\psi} = \oline{t_{Z(s_0)}}$ and the transfer maps are compatible with the \'{e}tale base change $f : f^{-1}(V) - f^{-1}(V \cap Z_2) \to V - V \cap Z_2$ (in particular $N$-triviality is preserved by \'etale base change), it suffices to show $\oline{t_{Z({f_1}^*(s_0))} } = \oline{ t_{\Delta_{f^{-1}(V) - f^{-1}(V \cap Z_2)}}}$. 

The argument is the same: the zero scheme of
$$t \cdot  {f_1}^*(s_0) + (1-t) \cdot s_{\Delta_{f^{-1}(V) - f^{-1}(V \cap Z_2)}} \cdot u_1$$
in $\A_t \times ({f^{-1}(V) - f^{-1}(V \cap Z_2)}) \times_S \oline{X_1}$ gives the relation
$$t_{Z({f_1}^*(s_0))} = t_{\Delta_{f^{-1}(V) - f^{-1}(V \cap Z_2)}} + t_{Z(u_1)} :  \cal{F}(X_1 - Z_1) \to \cal{F}({f^{-1}(V) - f^{-1}(V \cap Z_2)}).$$

Furthermore $Z(u_1)$ is closed in $f^{-1}(V) \times_S (X_1 - Z_1)$, i.e., extends through $f^{-1}(V \cap Z_2)$.  Therefore $\oline{  t_{Z(u_1)} } : \cal{F}(X_1 - Z_1) \to \cal{F}({f^{-1}(V) - f^{-1}(V \cap Z_2)}) / \cal{F}(f^{-1}(V))$ is zero and $\oline{t_{Z({f_1}^*(s_0))} } = \oline{ t_{\Delta_{f^{-1}(V) - f^{-1}(V \cap Z_2)}}}$, as desired.

\end{proof}

As a result of Theorem $\ref{general Nis invariance}$ we have the statement of \cite[Cor.~4.13]{CTPST} for $\cal{F}$ as in the theorem.

\begin{corollary} \label{V 4.13} Let $k$ be a perfect infinite field, and let $\cal{F}$ be a homotopy invariant presheaf on $\Sm /k$ with oriented weak transfers for affine varieties.  Let $f : X_1 \to X_2$ be an \'etale morphism of smooth $k$-schemes and $Z \into X_2$ a smooth divisor such that $f^{-1}(Z) \to Z$ is an isomorphism.  Then the canonical morphism of sheaves
$$\cal{F}_{(X_2, Z)} \to \cal{F}_{(X_1, f^{-1}(Z))}$$
on $Z_{Zar}$ is an isomorphism. \end{corollary} 

Then we have the analogue of \cite[Thm.~4.14]{CTPST}, \cite[Thm.~23.12]{MVW}.

\begin{theorem} \label{V 4.14} Let $k$ be a perfect infinite field, and let $\cal{F}$ be a homotopy invariant presheaf on $\Sm /k$ with oriented weak transfers for affine varieties.  Let $X$ be a smooth $k$-scheme and $Z \into X$ a smooth divisor.  Then about any $x \in X$ there is an open neighborhood $U$ and isomorphisms
$$\cal{F}_{(U \times Y, (U \cap Z) \times Y)} \cong \cal{F}_{(\A \times (U \cap Z) \times Y, (U \cap Z) \times Y)}$$
of sheaves on ${((U \cap Z) \times Y)}_{Zar}$ for all $Y \in \Sm /k$, natural in $Y$.
\end{theorem}

\begin{proof} We use the constructions of \cite[Thm.~4.14]{CTPST}, then apply Corollary $\ref{V 4.13}$ to the arrows in the diagram appearing at the end of the proof of \cite[Thm.~4.14]{CTPST}.  \end{proof}

\begin{corollary} \label{V 4.14 -1} In the situation of Theorem $\ref{V 4.14}$ there are isomorphisms
$$\cal{F}_{(U \times Y, (U \cap Z) \times Y)} \cong {({\cal{F}}_{-1})}_{Zar} .$$
\end{corollary}
\begin{proof} Use Proposition $\ref{V 4.11}$. \end{proof}

\begin{definition} A presheaf $\cal{F}$ \textit{has contractions} if it satisfies the conclusion of $\ref{V 4.14 -1}$.  By $\cal{F}_{-1}$ we denote the presheaf whose value on a scheme $X$ is $\coker(\cal{F}(X \times \A) \to \cal{F}(X \times (\A -0)) )$. \end{definition}

\begin{corollary}[see \cite{MVW} Lemma 22.10] \label{semilocal MV2} Let $k$ be a perfect infinite field, and let $\cal{F}$ be a homotopy invariant presheaf of abelian groups on $\Sm / k$ with oriented weak transfers for affine varieties.  Let $S$ be the semilocal scheme of a finite set of points on a smooth $k$-scheme, and let $S = U_0 \cup V$ be a Zariski open cover.  Then there exists an open $U \subset U_0$ such that $S = U \cup V$ and the Mayer-Vietoris sequence:
$$0 \to \cal{F}(S) \to \cal{F}(U) \oplus \cal{F}(V) \to \cal{F}(U \cap V) \to 0$$
is exact. \end{corollary}

\begin{proof} If $k$ is infinite, all of the constructions entering into Proposition $\ref{Nis corr}$ and Theorem $\ref{general Nis invariance}$ work for $z_2$ replaced by a finite set of points, and the neighborhood $V$ of $z$ can be made to contain any given finite set of points (not necessarily on the divisor $Z_2$).  If the sections are chosen sufficiently generic, the incidences of their zero schemes will avoid a finite set.   

Let $Z$ be a smooth divisor passing through the (finitely many) closed points not contained in $U$; this is a closed subscheme of $V$ which is also closed in $S$.  Then $U := S - Z \subset U_0$ is an open subscheme of $U_0$ such that $U \cup V = S$.  Furthermore we have $V - Z = V \cap (X - Z) = V \cap U$.

We apply Theorem $\ref{general Nis invariance}$ with $f : X_1 \to X_2$ the open immersion $j: V \subset S$ and $Z_2 = Z$.  Since $Z$ is a closed subscheme of $V$ which remains closed in $S$, the morphism $j^{-1}Z \to Z$ is an isomorphism.  Hence we get a map $\psi : \cal{F}(V - Z) \to \cal{F}(S - Z)$ inducing an isomorphism $\cal{F}(V - Z) / \cal{F}(V) \cong \cal{F}(S-Z) / \cal{F}(S)$ inverse to the natural map, hence the natural restriction map $\cal{F}(U) / \cal{F}(S) \to \cal{F}(U \cap V) / \cal{F}(V)$ is also an isomorphism.  Thus we may apply the 9-lemma to the diagram (whose border of outer zeroes has been omitted):

$$\xym{
0 \ar[r] \ar[d]  & \cal{F}(V) \ar[r]^-= \ar[d]^-{0,1} & \cal{F}(V) \ar[d] \\
\cal{F}(S) \ar[r] \ar[d]^-= & \cal{F}(U) \oplus \cal{F}(V) \ar[d]^-{p_1} \ar[r] & \cal{F}(U \cap V) \ar[d] \\
\cal{F}(S) \ar[r] & \cal{F}(U)  \ar[r] & \cal{F}(U) / \cal{F}(S) \\ }$$
to conclude the middle row is exact. \end{proof}

We can compute the Nisnevich cohomology of $U \subset \A_k$ with coefficients in a presheaf with oriented weak transfers.

\begin{proposition} \label{V 5.4} \label{MVW 22.7} Let $k$ be a field of characteristic $0$, and let $\cal{F}$ be a homotopy invariant presheaf of abelian groups on $\Sm / k$ with oriented weak transfers for affine varieties.  Then for any open subscheme $U \subset \A_k$, we have $H^0_{Nis}(U, \cal{F}_{Nis} ) = \cal{F}(U)$ and $H^i_{Nis}(U, \cal{F}_{Nis}) = 0 $ for $i \neq 0$. \end{proposition}

\begin{proof} By Theorem $\ref{open affine line MV}$ we have $ \cal{F} \cong \cal{F}_{Zar}$ on $U_{Zar}$.  Since any Nisnveich cover contains a dense open immersion (to cover the generic point), Corollary $\ref{semilocal inj}$ implies the map $\cal{F}(U) \to \cal{F}_{Nis}(U)$ is injective.  Hence we may assume $\cal{F}$ is separated as a presheaf on $U_{Nis}$.

Now Theorem $\ref{general Nis invariance}$ implies that given any $a \in \cal{F}_{Nis}(U)$ and any closed point $x \in U$, there exists an open $x \in V \subset U$ such that $a |_V \in \cal{F}(V) \subset \cal{F}_{Nis}(V)$.  Thus Zariski locally, $a$ is induced by sections in $\cal{F}$, and since $\cal{F}$ is a Zariski sheaf, we get an element of $\cal{F}(U)$ inducing $a$.  Thus the map $\cal{F}(U) \to \cal{F}_{Nis}(U)$ is also surjective.

Since $U$ is a curve only the case $i=1$ remains, and Theorem $\ref{general Nis invariance}$ produces the Mayer-Vietoris sequence for a Nisnevich cover of $U$ (via the method of Corollary $\ref{semilocal MV2}$).
\end{proof}

\section{Homotopy invariance of cohomology}
We carry out the rest of Voevodsky's argument using certain simplifications due to Mazza-Voevodsky-Weibel \cite{MVW}.  
Observe that we could have worked exclusively with affine varieties: if we know the (functorial) agreement of Zariski and Nisnevich cohomology on affine varieties, then we obtain an isomorphism of \v{C}ech to derived functor spectral sequences.  Similarly, homotopy invariance of cohomology on the category of affine varieties implies homotopy invariance of cohomology on all of $\Sm / k$.

\begin{proposition} \label{V 4.21} Let $k$ be a field of characteristic $0$, let $\cal{F}$ be a homotopy invariant presheaf of abelian groups on $\Sm / k$ with oriented weak transfers for affine varieties, and let $X$ be a smooth $k$-scheme.  Then $\cal{F}_{Zar}(X \times \A) \cong \cal{F}_{Zar}(X)$. \end{proposition}
\begin{proof} We may assume $X$ is integral with generic point $\eta$.  The zero section $0_X : X \to \A_X$ induces a splitting $\cal{F}_{Zar}(\A_X) = \cal{F}_{Zar}(X) \oplus \ker (0_X^*)$ compatible with the splitting $\cal{F}_{Zar}(\A_\eta) = \cal{F}_{Zar}(\eta) \oplus \ker (0_\eta^*)$.  By Corollary $\ref{V 4.19}$ the map $\cal{F}_{Zar}(\A_X) \to \cal{F}_{Zar}(\A_\eta)$ is injective, hence the maps $\cal{F}_{Zar}(X) \to \cal{F}_{Zar}(\eta)$ and $\ker (0_X^*) \to \ker (0_\eta^*)$ are both injective.  By Corollary $\ref{V 4.16}$ we have $\cal{F}(\eta) \cong \cal{F}_{Zar}(\A_\eta)$, hence $ \ker (0_\eta^*)$ is trivial.  But then $\ker (0_X^*)$ must be trivial as well, and the result follows. \end{proof}

\begin{notation} For a presheaf $\cal{F}$ of abelian groups on $\Sch /k$ we denote by $s_{Zar}(\cal{F})$ its Zariski separation.  On a $k$-scheme $X$ the value of $s_{Zar}(\cal{F})$ is $\cal{F}(X) / \cup (\ker (\cal{F}(X) \to \oplus_i \cal{F}(U_i) ))$, where the union is taken over the partially ordered set of Zariski covers $\{ U_i \}$ of $X$.  \end{notation}

\begin{proposition} \label{V 4.22} Let $k$ be a perfect infinite field, and let $\cal{F}$ be a homotopy invariant presheaf of abelian groups on $\Sm / k$ with (oriented) weak transfers for affine varieties.  Then $s_{Zar}(\cal{F})$ has a unique structure of homotopy invariant presheaf with (oriented) weak transfers for affine varieties such that the canonical morphism $\cal{F} \to s_{Zar}(\cal{F})$ is a morphism of presheaves with weak transfers for affine varieties. \end{proposition}
\begin{proof} The presheaf $s_{Zar}(\cal{F})$, being a quotient of a homotopy invariant presheaf, is homotopy invariant.  So suppose $f : X \to Y$ is a morphism for which a weak transfer $f_* : \cal{F}(X) \to \cal{F}(Y)$ exists and for which a weak transfer $f_* : s_{Zar}(\cal{F})(X) \to s_{Zar}(\cal{F})(Y)$ must be constructed.

We need to show that if $a \in \cal{F}(X)$ is a locally trivial element, then so is $f_*(a) \in \cal{F}(Y)$.  Since the weak transfers are compatible with open immersions, we may assume $Y$ is local.  Then $X$ is semilocal and all of the restriction maps $\cal{F}(X) \to \cal{F}(U_i)$ are injective by Corollary $\ref{semilocal inj}$.  But then already $a$ vanishes on the semilocal scheme $X$, hence $f_*(a)$ is zero as well.  \end{proof}

\begin{proposition} \label{V 4.24} \label{MVW 22.12}  Let $k$ be a perfect infinite field, let $\cal{F}$ be a homotopy invariant presheaf of abelian groups on $\Sm / k$ with oriented weak transfers for affine varieties, and let $S$ be a smooth semilocal $k$-scheme.  Then $\cal{F}_{Zar}(S) = \cal{F}(S)$. \end{proposition}
\begin{proof}  The canonical map $\cal{F}(S) \to s_{Zar}(\cal{F})(S)$ is injective by Corollary $\ref{semilocal inj}$, hence an isomorphism, so by Proposition $\ref{V 4.22}$ we may assume $\cal{F}$ is a separated presheaf.  As in \cite[4.24]{CTPST}, we use induction on the number of open sets in a cover together with Corollary $\ref{semilocal MV2}$ to construct a candidate lift of an element $a \in \cal{F}_{Zar}(S)$.  Then the assumption of separatedness and Corollary $\ref{semilocal inj}$ show the candidate lift works. \end{proof}

\begin{proposition} \label{V 4.25, 4.26} \label{MVW 22.1} \label{MVW 22.15} Let $k$ be a field of characteristic $0$, and let $\cal{F}$ be a homotopy invariant presheaf of abelian groups on $\Sm / k$ with oriented weak transfers for affine varieties.  Then $\cal{F}_{Zar}$ regarded as a sheaf on $\Sm / k$ has a unique structure of presheaf with oriented weak transfers for affine varieties such that $\cal{F} \to \cal{F}_{Zar}$ is a morphism of presheaves with transfer structure.  Furthermore the presheaf $\cal{F}_{Zar}$ is homotopy invariant and has contractions. \end{proposition}
\begin{proof} We may assume $\cal{F}$ is separated by Proposition $\ref{V 4.22}$.  Consider a morphism $f : X \to Y$ in $\Sm / k$ along which we must construct a transfer, and an element $a \in \cal{F}_{Zar}(X)$.  Since $\cal{F}$ is separated and a finite cover of a local scheme is semilocal, by Proposition $\ref{V 4.24}$ we can find a cover $\{ U_i \}$ of $Y$ such that the restriction $a_i := a |_{f^{-1}U_i} \in \cal{F}_{Zar}(f^{-1}U_i)$ belongs to the subgroup $\cal{F}(f^{-1}U_i)$.  Let $b_i$ denote the element ${f |_{f^{-1}U_i}}_* (a_i) \in \cal{F}(U_i)$.  Since the weak transfers are compatible with the open immersions $U_{ij} \to U_i$, the elements $b_i$ and $b_j$ coincide on $U_i \cap U_j$, hence they glue to a global element $b \in \cal{F}_{Zar}(Y)$.

Homotopy invariance is the statement of Proposition $\ref{V 4.21}$.  Hence $\cal{F}_{Zar}$ has contractions by Theorem $\ref{V 4.14}$.
\end{proof}

\begin{lemma} \label{TRICAT 3.1.5} Let $k$ be a field, and let $\cal{F}$ be a presheaf of abelian groups on $\Sm / k$ with oriented weak transfers for affine varieties.  Then the Nisnevich sheafification $\cal{F}_{Nis}$ has a unique structure of a presheaf with oriented weak transfers making the canonical morphism $\cal{F} \to \cal{F}_{Nis}$ a morphism of presheaves with oriented weak transfers for affine varieties. \end{lemma}

\begin{proof} First we note $\cal{F}_{Nis}$ inherits the transfer structure from $\cal{F}$.  This follows more or less immediately from the fact that a finite cover of a Henselian local scheme is a disjoint union of Henselian local schemes.  The requirement that $\cal{F}$ be compatible with disjoint unions then determines the induced weak transfers $f_* : \cal{F}_{Nis}(X) \to \cal{F}_{Nis}(Y)$ along finite morphisms to suitably shrunken targets $Y$.  For details see the $i=0$ case of Lemma $\ref{V 5.3}$.  \end{proof}

We have the first comparison results.

\begin{theorem} \label{MVW 22.2} Let $k$ be a field of characteristic $0$, and let $\cal{F}$ be a homotopy invariant presheaf of abelian groups on $\Sm / k$ with oriented weak transfers for affine varieties.  Then $\cal{F}_{Zar} = \cal{F}_{Nis}$. \end{theorem}

\begin{proof} Since we have Lemma $\ref{TRICAT 3.1.5}$, we can follow the argument of \cite[22.2]{MVW}: the kernel and cokernel of the presheaf map $\cal{F} \to \cal{F}_{Nis}$ have the same transfer structure and have trivial Nisnevich sheafification.  It suffices to show they have trivial Zariski sheafification, so we may assume $\cal{F}_{Nis} = 0$.  By Proposition $\ref{V 4.25, 4.26}$ we have that $\cal{F}_{Zar}$ is homotopy invariant with oriented weak transfers for affine varieties, and for any presheaf we have $\cal{F}_{Nis} = {(\cal{F}_{Zar})}_{Nis}$, hence we may assume $\cal{F}$ is a Zariski sheaf.

So we need that  $\cal{F}(S) =0$ for $S$ local.  But Theorem $\ref{general Nis invariance}$ implies that a section which is trivialized by a Nisnevich cover must already be Zariski locally trivial. \end{proof}

\begin{corollary} \label{MVW 23.12} Let $k$ be a field of characteristic $0$.  Then Theorem $\ref{V 4.14}$ and Corollary $\ref{V 4.14 -1}$ hold with the Zariski topology replaced by the Nisnevich topology. \end{corollary}

\begin{proposition}[application to contractions] \label{MVW 23.5} Let $k$ be a field of characteristic $0$, and let $\cal{F}$ be a homotopy invariant presheaf of abelian groups on $\Sm / k$ with oriented weak transfers for affine varieties.  Then ${(\cal{F}_{Nis})}_{-1} \cong {(\cal{F}_{-1})}_{Nis}$. \end{proposition}

\begin{proof} Proposition $\ref{MVW 22.1}$ and Theorem $\ref{MVW 22.2}$ together imply that $\cal{F}_{Nis}$ is homotopy invariant with oriented weak transfers.  The presheaf $\cal{F}_{-1}$ inherits homotopy invariance and the transfer structure, since if $f : X \to Y$  admits an $N$-trivial embedding then so does $f \times \id_T : X \times T \to Y \times T$.  (We use this for $T = \A, \A - 0$.)  Thus we have a canonical map of homotopy invariant presheaves with oriented weak transfers ${(\cal{F}_{-1})}_{Nis} \to {(\cal{F}_{Nis})}_{-1}$.  To apply Corollary $\ref{MVW 11.2}$ to the kernel and cokernel of this map, we need to show that $\cal{F}_{-1}(\spec E) = {(\cal{F}_{Nis})}_{-1}(\spec E)$ for all fields $E \supset k$.  In other words, we need $\cal{F}(\A_E  - 0) / \cal{F}(\A_E) = \cal{F}_{Nis}(\A_E - 0) / \cal{F}_{Nis}(\A_E)$, which follows from Theorem $\ref{MVW 22.2}$ and Theorem $\ref{MVW 22.4}$. \end{proof}

Now we can follow \cite[Lect.~24]{MVW}.

\begin{lemma} \label{V 5.3} \label{MVW 13.4} Let $k$ be a field, and let $\cal{F}$ be a Nisnevich sheaf of abelian groups on $\Sm / k$ with oriented weak transfers for affine varieties.  Then the cohomology presheaves $H^n_{Nis}(- \cal{F})$ have oriented weak transfers. \end{lemma}

\begin{proof} Following \cite[Ex.~6.20]{MVW}, we observe the canonical flasque resolution of $\cal{F}$ consists of terms which may be endowed with oriented weak transfers.  More precisely, $\cal{F}$ admits a canonical injection into the sheaf $E_\cal{F}$ whose value on $X \in \Sm / k$ is $\prod_{x \in X} \cal{F} (\spec \cal{O}^h_{X,x})$, i.e., the product of the values of $\cal{F}$ on the Hensel local rings of the closed points in $X$.  (Then we take the cokernel and repeat.)  Given a finite flat morphism $f : X \to Y$ along which $E_\cal{F}$ should have an oriented weak transfer, the fiber $f^{-1}(\spec \cal{O}^h_{Y,y})$ splits as $\prod_{x \in f^{-1}(y)} \spec \cal{O}^h_{X,x}$.  Therefore we may identify $E_\cal{F}(X) = \prod_{x \in X} \cal{F} (\spec \cal{O}^h_{X,x}) = \prod_{y \in Y} \prod_{x \in f^{-1}(y)} \cal{F}(\spec \cal{O}^h_{X,x})$.  To define the morphism $f_* : E_\cal{F}(X) \to E_\cal{F}(Y)$, it suffices to define for every $y \in Y$ a morphism $\prod_{x \in f^{-1}(y)} \cal{F}(\spec \cal{O}^h_{X,x}) \to \cal{F}(\spec \cal{O}^h_{Y,y})$, and for this we take the sum of the weak transfers $f_* : \cal{F}(\spec \cal{O}^h_{X,x}) \to \cal{F}(\spec \cal{O}^h_{Y,y})$. \end{proof}

\begin{remark} Since the canonical flasque resolution can be constructed using only affine varieties, the Nisnevich cohomology of $\cal{F}$ has oriented weak transfers even if $\cal{F}$ itself has such transfers only for affine varieties.  Also note Lemma \ref{V 5.3} fails for the Zariski topology: we have a restriction morphism $\cal{F}(f^{-1}(\spec \cal{O}_{Y,y})) \to \prod_{x \in f^{-1}(y)} \cal{F} (\spec \cal{O}_{X,x})$ and a weak transfer $f_* : \cal{F}(f^{-1}(\spec \cal{O}_{Y,y})) \to \cal{F}(\spec \cal{O}_{Y,y})$, but there is no reason the morphism $f_*$ should extend to $\prod_{x \in f^{-1}(y)} \cal{F} (\spec \cal{O}_{X,x})$.  The lemma holds for the Zariski topology with the homotopy invariance hypothesis (in characteristic zero) by Proposition $\ref{MVW 22.1}$. \end{remark}

\begin{corollary} \label{cohomology is homotopy invariant} Let $k$ be a field of characteristic $0$, and let $\cal{F}$ be a homotopy invariant presheaf of abelian groups on $\Sm / k$ with oriented weak transfers for affine varieties.  Then $H^n_{Nis}(-, \cal{F}_{Nis})$ is a homotopy invariant presheaf with oriented weak transfers. \end{corollary}
\begin{proof} The proof of \cite[24.1-4]{MVW} goes through: the arguments are sheaf-theoretic and use properties of presheaves with transfers that we have verified extend to presheaves with oriented weak transfers.  The case $n=0$ follows from Proposition $\ref{V 4.25, 4.26}$ and Theorem $\ref{MVW 22.2}$.  Hence we are reduced to \cite[24.2]{MVW} by the same Leray spectral sequence argument.

The proof of \cite[24.2]{MVW} goes through in our setting since we have the computation of the Nisnevich cohomology for open subschemes of $\A_E$ by Proposition $\ref{MVW 22.7}$.  Thus we are reduced to proving \cite[24.3]{MVW} in our setting; copying notation, we need to show the composition
$$H^n_{Nis}(S \times \A, \cal{F}) \xrightarrow{\tau} H^n_{Nis}(S \times \A, j_* j^* \cal{F}) \xrightarrow{\eta} H^n_{Nis}(U \times \A, j^* \cal{F})$$
is injective.  The proof that $\eta$ is injective applies since we have Proposition $\ref{MVW 23.5}$ and Corollary $\ref{MVW 23.12}$.  The proof that $\tau$ is injective also works in our setting.  Corollary $\ref{MVW 22.8}$ gives the injectivity $\cal{F} \to j_* j^* \cal{F}$, and Corollary $\ref{MVW 23.12}$ allows us to identify the cokernel as a contraction.  Finally, the argument for \cite[24.4]{MVW} is sheaf-theoretic. \end{proof}

\begin{corollary} \label{Zar = Nis} Let $k$ be a field of characteristic $0$, and let $\cal{F}$ be a homotopy invariant presheaf of abelian groups on $\Sm / k$ with oriented weak transfers for affine varieties.  Then $H^n_{Nis}( - \cal{F}_{Nis}) = H^n_{Zar}( - \cal{F}_{Zar})$ for all $n$. \end{corollary}

\noi \textbf{The Gersten resolution.} We have developed enough machinery that we can also imitate the construction of the Gersten resolution.  Corollaries \ref{cohomology is homotopy invariant}, \ref{semilocal inj}, \ref{V 4.15}, and \ref{Zar = Nis} provide us with the analogues of \cite[24.1, 11.1, 22.7]{MVW}.  Theorem \ref{V 4.14} is the analogue of \cite[23.12]{MVW}, and \cite[2.10, 23.4, 23.7]{MVW} are sheaf-theoretic.

\begin{theorem} \label{Gersten} Let $k$ be a field of characteristic $0$, and let $\cal F$ be a homotopy invariant presheaf of abelian groups on $\Sm / k$ with oriented weak transfers for affine varieties.  Let $X \in \Sm / k$ be a smooth $k$-scheme of dimension $n$, and let $X^{(i)}$ denote the set of codimension $i$ points of $X$.  Then there is a canonical exact sequence on $X_{Zar}$:

$$0 \to \cal{F} \to  \bigoplus_{{x \in X^{(0)}}} {i_x}_* (\cal{F}) \to \displaystyle \bigoplus_{x \in X^{(1)}} {i_x}_* (\cal{F}_{-1}) \to \cdots \to \displaystyle \bigoplus_{x \in X^{(n)}} {i_x}_* (\cal{F}_{-n}) \to 0.$$
\end{theorem}

\appendix
\counterwithin{theorem}{section}   

\section{Trivializing normal bundles} 

We collect some basic lemmas about normal bundles.  Let $k$ be a field.  We say a morphism of smooth $k$-schemes $f : X \to Y$ \textit{admits an $N$-trivial embedding} if there is a closed embedding of $Y$-schemes $X \into Y \times \mathbb{A}^n$ such that the normal bundle $N_X(Y \times \mathbb{A}^n)$ to $X$ in $Y \times \mathbb{A}^n$ is trivial.  The existence of any embedding $X \into Y \times \mathbb{A}^n$ implies $f$ is affine.  If $f$ is not itself an embedding, the existence of an embedding $X \into Y \times \AAA^n$ implies $X$ admits a nonconstant regular function.  If $f$ is \'{e}tale, then any embedding $X \into Y \times \mathbb{A}^n$ is $N$-trivial.

\begin{lemma} \label{N trivial no kahler} Let $f : X \to Y$ be a morphism in $\Sm / k$ such that $\Omega^1_{X/k}$ and $\Omega^1_{Y/k}$ are trivial.  Then $f$ admits an $N$-trivial embedding.
\end{lemma}
\begin{proof} For any embedding $i : X \to Y \times \bb{A}^n$ with ideal sheaf $\cal{I}$, in the exact sequence
$$0 \to {\cal{I} / {\cal{I}}^2} \to i^\ast \Omega^1_{Y \times \mathbb{A}^n /k} \to\Omega^1_{X/k}  \to 0$$
there is a section $\Omega^1_{X/k} \to i^\ast \Omega^1_{Y \times \mathbb{A}^n /k }$ since both of these bundles are trivial.  Hence $N_X(Y \times \bb{A}^n) \oplus \cal{O}^{\dim X} \cong \cal{O}^{n + \dim Y}$.  Therefore in any embedding $X \stackrel{i}{\rightarrow} Y \times \bb{A}^n \into Y \times \bb{A}^{n + \dim X}$ induced by a constant map $X \to \bb{A}^{\dim X}$, the normal bundle is trivial.
\end{proof}

\begin{lemma} \label{N trivial A1 stable}Let $f: X \to Y$ be a morphism in $\Sm / k$ that admits an $N$-trivial embedding.  Then any morphism of $Y$-schemes $X \to \A \times Y$ admits an $N$-trivial embedding.
\end{lemma}
\begin{proof} Suppose $i = (f, \beta) : X \into Y \times \bb{A}^n$ is an $N$-trivial embedding, and $(f, \alpha) : X \to Y \times \A$ is the given morphism.  Then we claim $(f, \alpha, \beta) : X \into Y \times \A \times  \bb{A}^n$ is $N$-trivial.   We denote by ${N}^\vee_{(-)}$ the conormal bundle of the embedding $(-)$.  We have an exact square:

$$\xym{{N}^\vee_{f, \beta} \ar[r] \ar[d] & f^\ast \Omega^1_{Y/k} \oplus \beta^\ast \Omega^1_{\bb{A}^n /k} \ar[r] \ar[d] & \Omega^1_{X/k} \ar[d]_-= \\
{N}^\vee_{f, \alpha, \beta} \ar[d] \ar[r] & f^\ast \Omega^1_{Y/k} \oplus \alpha^\ast \Omega^1_{\bb{A}^1 /k} \oplus \beta^\ast \Omega^1_{\bb{A}^n / k}\ar[r] \ar[d] & \Omega^1_{X/k}  \ar[d] \\
Q \ar[r]^-\cong &  \alpha^\ast \Omega^1_{\bb{A}^1 /k}  \ar[r] & 0 \\ }$$

The projection ${N}^\vee_{f, \alpha, \beta} \to  f^\ast \Omega^1_{Y/k} \oplus \beta^\ast \Omega^1_{\bb{A}^n /k}$ vanishes in $\Omega^1_{X/k}$, hence factors through ${N}^\vee_{f, \beta}$, so the left hand column is split exact.  Now ${N}^\vee_{f, \beta}$ is trivial by hypothesis, and $\Omega^1_{\bb{A}^1 /k}$ is trivial, hence ${N}^\vee_{f, \alpha, \beta}$ is trivial. \end{proof}

\begin{lemma} \label{N trivial smooth BC}
Let $f: X \to Y$ be a morphism in $\Sm / k$ that admits an $N$-trivial embedding.  Let $Y' \to Y$ be a flat morphism.  Then $f': X' \to Y'$ admits an $N$-trivial embedding.
\end{lemma}
\begin{proof} By \cite[B.7.4]{Ful}, the normal bundle to $X'$ in $Y' \times \bb{A}^n$ is simply the pull-back of the normal bundle to $X$ in $Y \times \bb{A}^n$. \end{proof}

\begin{corollary} \label{prin divisor}
Let $f: X \to Y$ be a flat morphism in $\Sm / k$, $X \to Y \times \bb{A}^n$ an $N$-trivial closed embedding (of $Y$-schemes), and $Y' \into Y$ a closed immersion in $\Sm / k$.  Suppose that $X' =X \cap (Y' \times \bb{A}^n)$ is smooth.  Then $X' \into Y' \times \bb{A}^n$ is an $N$-trivial embedding (of $Y'$-schemes).
\end{corollary}

\begin{proof} Note the flatness of $f$ implies $X$ and $Y' \times \bb{A}^n$ intersect properly in $Y \times \bb{A}^n$.
The embedding $X'  \into Y' \times \bb{A}^n$ is regular, hence so is the embedding $X' \into Y \times \bb{A}^n$.  Then \cite[B.7.4]{Ful} asserts that we have an exact sequence:
$$N_{X'}(Y \times \bb{A}^n) = N_{X}(Y \times \bb{A}^n) \arrowvert_{X'} \oplus N_{Y' \times \bb{A}^n}(Y \times \bb{A}^n) \arrowvert_{X'}.$$

We have also the exact sequence:
$$0 \to N_{X'}(Y' \times \bb{A}^n) \to N_{X'}(Y \times \bb{A}^n) \to N_{Y' \times \bb{A}^n}(Y \times \bb{A}^n) \arrowvert_{X'} \to 0.$$
Therefore this last sequence is split exact and $N_{X'}(Y' \times \bb{A}^n) = N_{X}(Y \times \bb{A}^n) \arrowvert_{X'} $.   \end{proof}

If a morphism of $k$-schemes $f: X \to Y$ admits an $N$-trivial embedding $X \into Y \times \bb{A}^n$, then we say a morphism $g: Y' \to Y$ is transversal to $f$ if the canonical morphism $N_{X'} (Y' \times \bb{A}^n) \to g'^* (N_X(Y \times \bb{A}^n))$ is an isomorphism.  This is independent of the choice of $N$-trivial embedding since it is equivalent to the canonical morphism $\Omega^1_{Y'/Y} |_{X'} \to \Omega^1_{X'/X}$ being an isomorphism.  In the application, we require the weak transfers along a finite flat morphism $f : X \to Y$ in $\Sm / k$, whose definition involves the choice of an $N$-trivial embedding, to be compatible with certain types of base change.  By Corollary \ref{prin divisor}, any base change $g : Y' \to Y$ is transversal to such an $f$.  We only needed compatibility with base change by smooth morphisms, and by inclusions of a smooth divisor.

\bibliography{tspecreferences}{}
\bibliographystyle{plain}

\end{document}